\documentclass[12pt,a4paper,oneside,onecolumn,fleqn]{article}
\usepackage{mathrsfs}
\usepackage{amsfonts}
\usepackage{latexsym}
\usepackage{amsmath,amsthm}
\usepackage{amssymb}
\usepackage{epsf}
\usepackage{graphicx}
\usepackage{indentfirst}
\usepackage{cite}
\usepackage{color}
\usepackage{diagbox}
\usepackage{booktabs}
\usepackage{multirow}
\usepackage{caption}

\renewcommand{\arraystretch}{1}

\theoremstyle{plain}
\newtheorem{thm}{\bf Theorem}[section]
\newtheorem{con}[thm]{\bf Construction}
\newtheorem{cor}[thm]{\bf Corollary}

\newtheorem{lem}[thm]{\bf Lemma}

\theoremstyle{definition}

\newtheorem{exam}{\bf Example}

\theoremstyle{remark}

\title{\bf The existence spectrum of near triple arrays with four rows}
\author{ Guangzhou Chen$^{1*}$, Yaxin Yue$^1$, Yong Zhang$^{2}$\\
\small1. {\it School of Mathematics
and Statistics, Henan Normal University,}\\
\small {\it Xinxiang, 453007, P. R. China }\\
\small2. {\it School of Mathematics and Statistics, Yancheng Teachers University,}\\
\small {\it Yancheng 224002, P. R. China}\\
}
\date{}
\begin{document}
\maketitle
\begin{center}
\begin{minipage}{16cm}
\begin{center}
{\bf Abstract}
\end{center}

\vspace{0.2cm} \ \ \ \
In the 1950s and 1960s, Agrawal introduced a class of experimental designs that  later became known as triple arrays. Gordeev, Markstr\"{o}m and \"{O}hman proposed near triple arrays by relaxing all three intersection properties of triple arrays, allowing two values concentrated around the average intersection size, as well as two consecutive values for the replication numbers. They completely resolved the existence of near triple arrays with three rows, showing that there exists a $(3\times c,v)$-near triple array if and only if $v\geq c\geq 3$ except for $(c,v)\in\{(3,6),(4,6),(5,8)\}$. In this paper, we further investigate the existence of near triple arrays with four rows and prove that there exists a $(4\times c,v)$-near triple array  if and only if $v\geq c\geq 4$ except for $(c,v)\in\{ (4,9),(5,7),(5,10),(6,8),(7,9),(10,12),(11,13)\}$.

\vskip 6pt
{\textit{Keywords:}} Triple arrays; Near triple arrays; Balanced incomplete block design; Maximally balanced maximally uniform designs; Modular Golomb ruler

\end{minipage}
\end{center}

{\begingroup\makeatletter  \let\@makefnmark\relax  \footnotetext{ 
\mbox{}\hspace{0.0in} $^*$Corresponding author: G. Chen (chenguangzhou0808@163.com)}
\vskip 0.5cm

\section{Introduction}

The study of experimental designs that allow for eliminating the influence of multiple factors on an experiment began with the early works of Fisher and was developed further by among others Agrawal in the 1950s and 1960s.
Agrawal \cite{Agrawal1966} introduced a class of experimental designs that later became known as \emph{triple arrays}. An $(r\times c, v)$-\emph{triple array} is an $r\times c$ array on $v$ symbols that is \emph{binary} (no symbol appears more than once in any row or column), \emph{equireplicate} (each symbol occurs exactly $e$ times), and satisfies the following three intersection conditions:

\noindent (RR) any two distinct rows contain $\lambda_{rr}$ common symbols;

\noindent (CC) any two distinct columns contain $\lambda_{cc}$ common symbols;

\noindent (RC) any row and column contain $\lambda_{rc}$ common symbols.

We denote a triple array as $(v,e,\lambda_{rr},\lambda_{cc},\lambda_{rc}:r\times c)$-TA, or simply $(r\times c,v)$-TA. An array satisfying conditions (RR) and (CC) is defined as \emph{a double array}, denoted by $(v,e,\lambda_{rr},\lambda_{cc}:r\times c)$-DA, or abbreviated as $(r\times c,v)$-DA. Condition (RC) is often referred to as \emph{adjusted orthogonality}, and relevant details can be found in
Sections 8 and 13 of Bailey's survey \cite{Bailey2017}. Bagchi and Shah\cite{Bagchi-Shah1989} showed that triple arrays are statistically optimal among all binary equireplicate arrays with respect to a broad class of optimality criteria.
Agrawal \cite{Agrawal1966} proposed a construction of extremal triple arrays based on a symmetric $2$-design and provided examples for a range of parameter sets. 
The general problem of selecting representatives from prescribed cell sets, with distinct representatives within each row and within each column, is NP-complete, as shown by Fon-Der-Flaass \cite{Fon-Der-Flaass-1997}.
Preece, Wallis and Yucas \cite{Preece2005} utilized field theory to construct $(q\times(q+1), 2q)$-TAs for all odd prime powers $q\geq 5$, referred to as \emph{Paley triple arrays}. Nilson and \"{O}hman \cite{Nilson2015} proposed a method of constructing triple arrays based on Youden rectangles and, in particular, verified that all Paley triple arrays can be obtained via this approach. Nilson and Cameron \cite{Nilson2017} further investigated this method for Youden rectangles derived from difference sets and constructed $((2u^2-u)\times (2u^2 + u), 4u^2- 1)$-TAs for positive integers $u$ whose square-free part divides $6$.
Bagchi and Bagchi \cite{Bagchi2026} constructed $((q+1)\times q^2,q(q+1))$-TAs for all odd prime powers $q\geq 3$ by using ovals in finite projective planes. These arrays, together with the above known triple arrays, form the only three known infinite families. For additional relevant findings, see \cite{Seberry1979, Street1981, Bagchi1998}.

Unfortunately, the intersection conditions (RC), (RR) and (CC) impose highly restrictive constraints on the possible sizes of such designs. Recently, Gordeev, Markstr\"{o}m and \"{O}hman \cite{Gordeev2026} introduced near triple arrays by relaxing all three intersection properties of triple arrays to allow two values concentrated around the average intersection size, as well as two consecutive values for the replication numbers.

An $r\times c$ \emph{row-column design} on $v$ symbols is a two-dimensional array with $r$ rows and $c$ columns, where each cell is filled with one of the $v$ symbols. It is \emph{binary} if no symbol appears more than once in any row or column. For a row-column design to be binary, the symbol number $v$ evidently satisfies max$(r,c)\leq v$, as $v$ cannot be smaller than either the row count $r$ or column count $c$. On the other hand, if $v$ exceeds the total number of cells, then some symbols are not used at all and then the binary row-column design always exists. Therefore, it suffices for us to consider only the existence of the binary row-column design with $v\leq rc$. Moreover, to avoid trivial examples, we generally consider only $r,c\geq3$.

The average replication number of a row-column design is given by $e=\frac{rc}{v}$. Let $e^-=\lfloor e\rfloor$ and $e^+=\lceil e\rceil$. If $e$ is an integer and every symbol occurs $e$ times in the array, the row-column design is called \emph{equireplicate} with replication number $e$. We will call a row-column design \emph{near equireplicate} if $e$ is not an integer and every symbol appears either $e^-$ or $e^+$ times. For near equireplicate designs, the number of occurrences of the symbols can be counted as described in the following lemma (see Lemma 2.1 in \cite{Gordeev2026}).

\begin{lem}[\!\!\cite{Gordeev2026}, Lemma 2.1]
In a near equireplicate $r\times c$ row-column design on $v$ symbols, there are $v^-=v(e^+-e)$ symbols occurring $e^-$ times and $v^+=v(e-e^-)$ symbols occurring $e^+$ times.
\end{lem}

For a binary $r\times c$ row-column design on $v$ symbols,
we denote by $\lambda_{rr}$, $\lambda_{cc}$, $\lambda_{rc}$ the average numbers of common symbols between two rows, between two columns, and between a row and a column, respectively. For proofs of the following lemma in more restrictive settings, see Theorem 2.2 and Theorem 3.1 in \cite{McSorley2005}, as well as Lemma 2.2 in \cite{Gordeev2026}.

\begin{lem}[\!\!\cite{McSorley2005}, Theorem 2.2 and Theorem 3.1; \cite{Gordeev2026}, Lemma 2.2]\label{lambda}
In a binary equireplicate $r \times c$ row-column design on $v$ symbols,
$\lambda_{rc} = e$, $\lambda_{rr} = \frac{c(e-1)}{r-1} = \frac{c(\lambda_{rc}-1)}{r-1}$, $\lambda_{cc} = \frac{r(e-1)}{c-1} = \frac{r(\lambda_{rc}-1)}{c-1}$.
\end{lem}

Let $\lambda_{rc}^- = \lfloor \lambda_{rc} \rfloor$, $\lambda_{rc}^+ = \lceil \lambda_{rc} \rceil$, and define $\lambda_{rr}^-$, $\lambda_{rr}^+$, $\lambda_{cc}^-$ and $\lambda_{cc}^+$ analogously. Gordeev et al. \cite{Gordeev2026} extended Lemma \ref{lambda} to the near equireplicate case.
They then proved the following result.

\begin{lem}[\!\cite{Gordeev2026}, Lemma 2.3]\label{Nec-par}
Let $v\leq rc$. In a binary (near) equireplicate $r \times c$ row-column design on $v$ symbols,

(1) $\lambda_{rc} = e^- + e^+ - \frac{e^-e^+}{e} = e + \frac{(e^+-e)(e-e^-)}{e}$.

(2) $\lambda_{rc}^- = e^- \text{ and } \lambda_{rc}^+ = e^+$.

(3) $\lambda_{rr} = \frac{c(\lambda_{rc}-1)}{r-1}$.

(4) $\lambda_{cc} = \frac{r(\lambda_{rc}-1)}{c-1}$.
\end{lem}

We now present the definition of a near triple array \cite{Gordeev2026}.
An $(r\times c,v)$-\emph{near triple array} is a binary equireplicate or near equireplicate $r\times c$ row-column design on $v$ symbols that satisfies the following three conditions:

1. any two rows have either $\lambda_{rr}^-$ or $\lambda_{rr}^+$ common symbols;

2. any two columns have either $\lambda_{cc}^-$ or $\lambda_{cc}^+$ common symbols;

3. any row and column have either $\lambda_{rc}^-$ or $\lambda_{rc}^+$ common symbols.

\noindent
An $(r\times c,v)$-near triple array is also denoted by $(v,\{e^-,e^+\},\{\lambda_{rr}^-,\lambda_{rr}^+\},\{\lambda_{cc}^-,\lambda_{cc}^+\},
\{\lambda_{rc}^-,\lambda_{rc}^+\}:r\times c)$-NTA, or abbreviated as $(r\times c,v)$-NTA. When $v\geq rc$, assigning mutually distinct symbols to $rc$ cells, with the remaining $v-rc$ symbols unused, yields an NTA. In this case, the replication number of each symbol is either $0$ or $1$, the intersection of any two rows or any two columns is $0$, and the intersection of any row and any column is $1$.

Gordeev et al. \cite{Gordeev2026} completely resolved the existence of near triple arrays with three rows and proved the following result.

\begin{thm}[\!\cite{Gordeev2026}, Theorem 5.10 and Table B.1] Let $c\geq 3$ and $v$ be positive integers.
There exists a $(3\times c,v)$-NTA if and only if $v\geq c$ except for $(c,v)\in\{(3,6),(4,6),(5,8)\}$.
\end{thm}

In this paper, we further investigate the existence of near triple arrays with four rows and establish the following theorem.

\begin{thm}\label{MainTH}
Let $c\geq 4$ and $v$ be positive integers. A $(4\times c,v)$-NTA exists if and only if $v\geq c$ except for $(c,v)\in\{ (4,9),(5,7),(5,10),(6,8),(7,9),(10,12),(11,13)\}$.
\end{thm}

The rest of the paper is organized as follows. In Section 2, we provide three constructions of near triple arrays, including via modular Golomb rulers, direct construction and matching designs. The proof of Theorem \ref{MainTH} is provided in Section 3. In Section 4, we summarize the conclusions of the paper and outline directions for future research.

\section{Construction methods}

In this section, we present three constructions of near triple arrays, focusing on those derived from modular Golomb rulers, direct construction, and matching designs. Full details of each construction are given below.

\subsection{Constructions for NTAs via modular Golomb rulers}

Let $G$ be an additively written group of order $n$. A $k$-subset $D$ of $G$ is an $(n, k, \lambda; u)$-\emph{difference set} of order $u = k - \lambda$ if every nonzero element of $G$ has exactly $\lambda$ representations as a difference $d - d'$ with elements from $D$. The difference set is \emph{abelian} or \emph{cyclic} if the group $G$ has the respective property. The redundant parameter $u$ is sometimes omitted; therefore, the notion of $(n, k, \lambda)$-difference sets is also used. Nilson and Cameron \cite{Nilson2017} established the connection between difference sets and triple arrays as follows.

\begin{lem}[\!\!\cite{Nilson2017}, Theorem 3.4]\label{DS-TA}
If there exists an $(n, k, \lambda)$-difference set in an abelian group $G$ that admits $-1$ as a multiplier, then there exists a $(k \times (n - k),n-1)$-TA.
\end{lem}

Subsequently, Nilson and Cameron utilized difference sets that satisfy the property of Lemma \ref{DS-TA} to develop a family of triple arrays, which is shown below \cite{Nilson2017}.

\begin{lem}[\!\!\cite{Nilson2017}, Theorem 4.8]
Let $u$ be a positive integer such that the square-free part of $u$ divides $6$, then there exists a $((2u^2 - u) \times (2u^2 + u),4u^2 - 1)$-TA.
\end{lem}

An $(n,k;v)$-\emph{difference packing} over $\mathbb{Z}_n$, or $v$-DP$(n,k)$,
is a pair $(\mathbb{Z}_n,\mathcal{X})$, where $\mathcal{X} = \{X_1,X_2,\dots,X_v\}$ is a family of $k$-subsets of $\mathbb{Z}_n$, and for $1 \leq i \leq v$, $X_i = \{a_{i1},a_{i2},\dots,a_{ik}\}\subseteq\mathbb{Z}_n$, such that all of the differences, $\{a_{ij} - a_{il} \mid 1 \leq i \leq v, 1 \leq j \neq l \leq k\}$, are distinct and nonzero modulo $n$. An $(n,k)$-\emph{modular Golomb ruler} (or $(n, k)$-MGR) is a $1$-DP$(n,k)$.

In what follows, we establish a number of basic results and standard constructions.

\begin{lem}[\!\!\cite{Buratti2021}, Theorem 1.1]\label{MGR=DS}
If there exists an $(n,k)$-MGR, then $n \ge k^2 - k + 1$. Further, a $(k^2 - k + 1, k)$-MGR is equivalent to a cyclic $(k^2 - k + 1, k, 1)$-difference set.
\end{lem}

Of course a $(q^2 + q + 1, q + 1, 1)$-difference set (i.e., a Singer difference set) is known to exist if $q$ is a prime power. So we have the following Corollary.

\begin{cor}[\!\!\cite{Buratti2021}, Corollary 1.2]
There exists a $(k^2 - k + 1, k)$-MGR if $k - 1$ is a prime power.
\end{cor}

It is widely conjectured that an abelian $(q^2 + q + 1, q + 1, 1)$-difference set exists only if $q$ is a prime power, and this conjecture has been verified for all $q < 2,000,000$ (see \cite{Gordon1994}). For prime numbers or prime powers, we have the following results on MGRs.

\begin{lem}[\!\!\cite{Bose1942}]\label{Prime-power}
For any prime power $q$, there is a $(q^2 - 1, q)$-MGR.
\end{lem}

\begin{lem}[\!\!\cite{Ruzsa1993}]\label{Prime}
For any prime $p$, there is a $(p^2 - p, p - 1)$-MGR.
\end{lem}

Difference sets provide a fundamental tool for constructing triple arrays (TAs). When $c=v$, this construction yields Youden rectangles (or Youden squares), which can be obtained by developing a cyclic difference set; see Construction 3.1 of \cite{Nilson2017}. Since modular Golomb rulers (MGRs) generalize cyclic difference sets, the following result naturally generalizes this construction, extending it to parameters not covered by ordinary difference sets. Moreover, when $c=v$, the following lemma is directly related to the near Youden rectangles introduced in \cite{JMSO23}, which are essentially near triple arrays with $c=v$. Thus, the following lemma also establishes the existence of near triple arrays for a variety of parameters.

\begin{lem}\label{MGR-NTA}
If there exists an $(n,k)$-MGR, then there exists a $(k\times n, n)$-NTA.
Further, if there exists an $(n,k)$-MGR with $n>k(k-1)+1$, then there exists a $(k\times (n-2), n)$-NTA.
\end{lem}
\begin{proof}
Suppose that $\{x_0,x_1,x_2,\cdots,x_{k-1}\}$ is an $(n,k)$-MGR over $\mathbb{Z}_n$.
Define a $k\times n$ array $A=(a_{i,j})$, $0\leq i\leq k-1$, $0\leq j\leq n-1$, where $a_{ij}=(x_i+j)\pmod n$. It is straightforward to verify that each symbol in $\mathbb{Z}_n$ occurs exactly once in each row, and at most once in each column, each symbol occurs $k$ times in the array $A$, and each pair of distinct columns has at most one symbol in common according to the property of an $(n,k)$-MGR. Therefore, $A$ is a $(k\times n, n)$-NTA.

Furthermore, if $n=k(k-1)+1$, then by Lemma \ref{MGR=DS}, each pair of distinct columns has exactly one symbol in common. Therefore, there must be two columns of $A$ that are disjoint when $n>k(k-1)+1$. Without loss of generality, assume that columns
$\alpha$ and $\beta$ of $A$ are disjoint. Define $A'$ as a $k\times (n-2)$ array by removing the two disjoint columns $\alpha$ and $\beta$ of $A$. It is clear that each symbol in $\mathbb{Z}_n$ occurs at most once in each row, and at most once in each column, each symbol occurs $k$ or $k-1$ times in $A'$; each pair of distinct columns has at most one symbol in common and each pair of distinct rows has exactly $n-4$ symbols in common in $A'$.

Let the two removed columns be the $j_1$-th and $j_2$-th columns. For any row $i_1$, the two deleted symbols take the forms $(x_{i_1}+j_1)\pmod n$ and $(x_{i_1}+j_2)\pmod n$. For any column $j$ of $A'$ with $j\neq j_1, j_2$, the symbol set of column $j$ is defined as $C_j=\{(x_i+j)\pmod n\mid i=0,1,\ldots,k-1\}$.

Suppose $(x_{i_1}+j_1)\pmod n, (x_{i_1}+j_2)\pmod n \in C_j$, then there exist
indices $i_1',i_2'\in\{0,1,\ldots,k-1\}$ such that

\begin{center}
$(x_{i_1}+j_1)\pmod n=(x_{i_1'}+j)\pmod n$ and $(x_{i_1}+j_2)\pmod n=(x_{i_2'}+j)\pmod n$.
\end{center}
Rearranging congruences gives $x_{i_1}-j\equiv x_{i_1'}-j_1\equiv x_{i_2'}-j_2\pmod n$, which further implies
$(x_{i_2'}+j_1)\pmod n=(x_{i_1'}+j_2) \pmod n$. This contradicts the fact that the two columns $j_1$ and $j_2$ are disjoint.
Consequently, any row and column have either $k-1$ or $k$ common symbols. It follows that $A'$ is a $(k\times (n-2), n)$-NTA.
This completes the proof.
\end{proof}

Suppose $T$ is an $(r \times n, n)$-NTA with $r \leq n$. Let $T'$ denote the row-column design obtained from $T$ by removing an arbitrary column from $T$.
Then $T'$ is an $(r \times (n - 1), n)$-NTA. This leads to the following conclusion.

\begin{lem}[\!\!\cite{Gordeev2026}, Lemma 5.7]\label{Del-Col}
If there exists an $(r \times n, n)$-NTA for $r \leq n$, then there exists an $(r \times (n-1), n)$-NTA.
\end{lem}

Combining Lemma \ref{MGR-NTA} and Lemma \ref{Del-Col}, we obtain the following lemma.
\begin{lem}\label{cor-n-1}
If there exists an $(n,k)$-MGR
for $k \leq n$, then there exists a $(k \times (n-1), n)$-NTA.
\end{lem}

Combining Lemmas \ref{Prime-power}-\ref{MGR-NTA} and Lemma \ref{cor-n-1}, we obtain the following lemma.
\begin{lem}\label{cor-n-2}
A $(q \times (q^2-1), q^2-1)$-NTA and a $(q \times (q^2-2), q^2-1)$-NTA exist for any prime power $q$, and a $((p-1) \times (p^2-p), p^2-p)$-NTA and a $((p-1) \times (p^2-p-1), p^2-p)$-NTA exist for any prime $p$.
\end{lem}

Given a positive integer $k \ge 3$, define

\begin{center}
MGR$(k) = \{n:$ there exists an $(n,k)$-MGR\}.
\end{center}
The results on modular Golomb rulers are summarised below.

\begin{lem}[\!\!\cite{Buratti2021}, Theorem 2.3]\label{Known-MGR}
(1) $\operatorname{MGR}(3) = \{n : n \ge 7\}$.

(2) $\operatorname{MGR}(4) = \{n : n \ge 13\}$.

(3) $\operatorname{MGR}(5) = \{21\} \cup \{n : n \ge 23\}$.

(4) $\operatorname{MGR}(6) = \{31\} \cup \{n : n \ge 35\}$.

(5) $\operatorname{MGR}(7) = \{n : n \ge 48\}$.

(6) $\operatorname{MGR}(8) = \{57\} \cup \{n : n \ge 63\}$.

(7) $\operatorname{MGR}(9) = \{73, 80\} \cup \{n : n \ge 85\}$.

(8) $\operatorname{MGR}(10) = \{91\} \cup \{n: n \ge 107\}$.

(9) $\operatorname{MGR}(11) = \{120, 133\} \cup \{n: n \ge 135\}$.
\end{lem}

\begin{lem}[\!\!\cite{Buratti2021}, Theorem 2.7]\label{Known-MGR2}
For any integer $k\geq 3$ and any integer $n\geq 3k^2-1$, there exists an $(n, k)$-MGR.
\end{lem}

\begin{lem}\label{4nn}
Both an $(r\times n, n)$-NTA and an $(r\times (n-1), n)$-NTA exist for any parameters $r$ and $n$ listed below.

\vskip 6pt
\centering
\begin{tabular}{|c c|c c|}
\hline
$r=4$ & $n\geq 13$ & $r=5$ & $n\geq 23$ or $n=21$ \\
$r=6$ & $n\geq 35$ or $n=31$ & $r=7$ & $n\geq 48$ \\
$r=8$ & $n\geq 63$ or $n=57$ & $r=9$& $n\geq 85$ or $n=73,80$ \\
$r=10$ & $n\geq 107$ or $n=91$ & $r=11$ & $n\geq 135$ or $n=120,133$ \\
$\forall$ $r$ & $n\geq 3r^2-1$ &  &  \\\hline
\end{tabular}
\end{lem}
\begin{proof}
By Lemma \ref{MGR-NTA} and Lemmas \ref{Known-MGR}-\ref{Known-MGR2}, an $(r\times n, n)$-NTA can be obtained. It then follows directly from Lemma \ref{Del-Col} that an $(r\times (n-1), n)$-NTA also exists.
\end{proof}

Combining Lemma \ref{MGR-NTA} and Lemma \ref{Known-MGR}, we derive the following lemma.
\begin{lem}\label{4n-2-n}
An $(r\times (n-2), n)$-NTA exists for any parameters $r$ and $n$ listed below.

\vskip 6pt
\centering
\begin{tabular}{|c c|c c|}
\hline
$r=4$ & $n\geq 14$ & $r=5$ & $n\geq 23$ \\
$r=6$ & $n\geq 35$ & $r=7$ & $n\geq 48$ \\
$r=8$ & $n\geq 63$ & $r=9$& $n\geq 85$ or $n=80$ \\
$r=10$ & $n\geq 107$ & $r=11$ & $n\geq 135$ or $n=120,133$ \\
$\forall$ $r$ & $n\geq 3r^2-1$ &  &  \\\hline
\end{tabular}
\end{lem}

\subsection{Direct constructions}

The following lemma, which directly generalizes Lemma 5.12 in \cite{Gordeev2026}, is clearly inspired by that result.

\begin{lem}\label{direct-Con}
(1) There exists a $(4\times 3k,4k+\eta)$-NTA for
$\eta\in\{0,1,2\}$ and every $k\geq5$ with $k\neq6$.

(2) There exists a $(4\times(3k+1),4k+2)$-NTA for every
$k\geq5$ with $k\neq6$.
\end{lem}
\begin{proof}
(1) Let $T=(t_{i,j})$, $1\leq i\leq 4$, $1\leq j\leq 3k$, be a $4\times3k$ row-column design on the point set $X=\mathbb Z_{4k}=\{0,1,\ldots,4k-1\}$.
For $0\leq a<k$ and $0\leq b<4$, place the symbol $4a+b$ in the following three cells: $(b+1,a+1)$, $((b+1)\pmod 4+1, (a+b)\pmod k+k+1)$, and $((b+2)\pmod 4+1, (a+2b)\pmod k+2k+1)$. Thus, for each symbol $4a+b$, its three occurrences lie in the three column groups
\begin{center}
$C_0=\{1,\ldots,k\},\ C_1=\{k+1,\ldots,2k\}, \
C_2=\{2k+1,\ldots,3k\}$.
\end{center}
Clearly, every symbol occurs exactly once in each of the three
column groups, and hence exactly three times in $T$. Therefore, $T$
is equireplicate, and each symbol appears at most once in each column.
Let $\mathcal R_i$, $1\leq i\leq 4$, denote the set of elements in the $i$-th row.
Then
\begin{center}
$\mathcal R_1=\{4a+l\mid0\leq a<k,\ l=0,2,3\}$,\ \
$\mathcal R_2=\{4a+l\mid0\leq a<k,\ l=0,1,3\}$,
\end{center}
\begin{center}
$\mathcal R_3=\{4a+l\mid0\leq a<k,\ l=0,1,2\}$,\ \
$\mathcal R_4=\{4a+l\mid0\leq a<k,\ l=1,2,3\}$.
\end{center}
It follows that $|\mathcal R_i|=3k$ for each $1\leq i\leq 4$. Thus, each symbol appears at most once in each row, and hence $T$ is binary. It is straightforward to verify that $|\mathcal R_{i_1}\cap \mathcal R_{i_2}|=2k$ for any $1\leq i_1\neq i_2\leq 4$.

We now verify the intersections between columns. To facilitate computations modulo
$k$, identify the columns in each group with
$\mathbb Z_k$. Then the three columns containing $4a+b$ are
\begin{center}
$a,\ \ (a+b)\pmod k,\ \ (a+2b)\pmod k$.
\end{center}
Let $C_{p,x}$ denote the column corresponding to $x\in\mathbb Z_k$
in the $p$-th column group, where $p\in\{0,1,2\}$. Then
\begin{center}
$4a+b\in C_{p,x}
\quad\Longleftrightarrow\quad
a+pb\equiv x\pmod k$.
\end{center}
Consequently, if $4a+b\in C_{p,x}\cap C_{q,y}$, then
\begin{center}
$a+pb\equiv x\pmod k,\ \ a+qb\equiv y\pmod k$.
\end{center}
Subtracting these congruences gives $(p-q)b\equiv x-y\pmod k$.
We consider the three possible pairs of distinct column groups.

If $\{p,q\}=\{0,1\}$, then $b\equiv \pm(x-y)\pmod k$.
Since $b\in\{0,1,2,3\}$ and $k\ge5$, there is at most one possible
value of $b$ satisfying $b\equiv \pm(x-y)\pmod k$. Once $b$ is determined, the congruence $a+pb\equiv x\pmod k$ uniquely determines $a$. Therefore, $|C_{p,x}\cap C_{q,y}|\le1$.

Similarly, if $\{p,q\}=\{1,2\}$, then we also have $b\equiv \pm(x-y)\pmod k$,
and hence $|C_{p,x}\cap C_{q,y}|\le1$.

It remains to consider $\{p,q\}=\{0,2\}$. In this case, the above system of congruences yields $2b\equiv \pm(y- x)\pmod k$. Suppose that two distinct symbols $4a+b$ and $4a'+b'$ belong to
both $C_{0,x}$ and $C_{2,y}$. Then $2b\equiv2b'\pmod k$. Since $b,b'\in\{0,1,2,3\}$ and $b\neq b'$, we have $k\mid2(b-b')$. Now $0<|2(b-b')|\le6$.
Because $k\ge5$, the only possible value of $k$ dividing
$2(b-b')$ is $k=6$. Indeed, for $k=6$, taking $|b-b'|=3$ gives $2(b-b')\equiv0\pmod6$. For every $k\ge5$, $k\neq6$, this is impossible. Hence
$|C_{0,x}\cap C_{2,y}|\le1$ whenever the two columns are distinct.

Two distinct columns belonging to the same column group
cannot contain a common symbol, because every symbol occurs exactly
once in each column group. Thus, any two distinct columns have at most one common symbol.

Finally, we prove that every row and every column of \(T\) have exactly three common symbols. Fix a symbol $\alpha=4a+b$, where $0\le a<k,\ 0\le b<4$.
By the construction, \(\alpha\) occurs in the following three rows:
$b+1,\ (b+1)\pmod 4+1,\ (b+2)\pmod 4+1$. More explicitly,
$$\begin{array}{c|c}
b & \text{rows containing }4a+b\\ \hline
0 & 1,2,3\\
1 & 2,3,4\\
2 & 3,4,1\\
3 & 4,1,2
\end{array}$$
Hence, each symbol occurs in exactly three rows and is absent from exactly one row. The unique row in which \(4a+b\) is absent is
$$\begin{cases}
4,&b=0,\\
1,&b=1,\\
2,&b=2,\\
3,&b=3.
\end{cases}$$
Thus, as \(b\) runs through \(0,1,2,3\), the four symbols \(4a+b\) are respectively absent from the four different rows.
Now fix an arbitrary column \(C_{p,x}\), where \(p\in\{0,1,2\}\) and \(x\in\mathbb Z_k\). From the construction, the symbol \(4a+b\) occurs in \(C_{p,x}\) precisely when $a+pb\equiv x\pmod k$. In fact, for the first group, that is, \(p=0\), the column coordinate is \(a\); for the second group \(p=1\), the column coordinate is \(a+b\); for the third group \(p=2\), the column coordinate is \(a+2b\).
For each fixed \(b\in\{0,1,2,3\}\), the congruence
$a+pb\equiv x\pmod k$ has a unique solution
$a\equiv x-pb\pmod k$.
Therefore, the column \(C_{p,x}\) contains exactly one symbol corresponding to each value of \(b\). In other words, there are exactly four symbols in \(C_{p,x}\), say
$x_b=4a_b+b,\ b=0,1,2,3$, where $a_b\equiv x-pb\pmod k$.
Hence, if \(\mathcal{R}_i\) is any fixed row, exactly one of the four symbols in \(C_{p,x}\) is absent from \(\mathcal{R}_i\). The other three symbols occur in \(\mathcal{R}_i\).
Therefore, $|\mathcal{R}_i\cap C_{p,x}|=4-1=3$, we conclude that every row and every column have exactly three common symbols. This proves the required row-column intersection condition.
It follows that, for every $k\ge5$, $k\neq6$, $T$ is a $(4\times3k,4k)$-NTA.

We next enlarge the symbol set by one symbol.

Let $T_1$ be obtained from $T$ by replacing the symbol $0$ in cell $(1,1)$
and the symbol $2$ in cell $(4,k+3)$ with a new symbol $4k$. Clearly,
$T_1$ is binary. The new symbol $4k$ occurs twice, whereas the symbols $0$ and $2$
occur twice and every other old symbol occurs three times. Thus,
$T_1$ is near equireplicate. The two columns containing these cells have exactly one common symbol in $T$, namely $2$; so replacing this occurrence of $2$
by $4k$ preserves the property that any two columns have at most
one common symbol. Similarly, the replacement of the occurrence
of $0$ does not create any new column intersection of size greater
than one. The elements in the first column of $T_1$ are $4k,1,2,3$ and
the elements in the $(k+3)$-th column of $T_1$ are $4k-1,4k,5,8$. It follows that
they intersect in exactly one element, that is, the number of intersections is $1$.
Therefore, the number of intersections between any two columns is either $0$ or $1$.

The modification changes the row intersections only by one, and a
direct inspection gives $|\mathcal R_1\cap\mathcal R_2|=|\mathcal R_1\cap\mathcal R_3|=|\mathcal R_3\cap\mathcal R_4|=2k-1$, while $|\mathcal R_2\cap\mathcal R_3|=|\mathcal R_1\cap\mathcal R_4|=|\mathcal R_2\cap\mathcal R_4|=2k$.
Thus the two possible row intersection numbers differ by one. The
row-column intersections remain equal to either two or three,
and the column-pair intersections are still at most one.
Consequently, $T_1$ is a $(4\times3k,4k+1)$-NTA.

Now obtain $T_2$ from $T_1$ by replacing the symbol $9$ in cell
$(2,3)$ and the symbol $11$ in cell $\bigl(1,(5\bmod k)+k+1\bigr)$
by a new symbol $4k+1$. The two selected cells lie in columns $3$ and $(5\bmod k)+k+1$, respectively. Denote a column by \((p,x)\), where \(p\in\{0,1,2\}\) and \(x\in\mathbb Z_k\); here \(p\) is the column group starting from \(0\), and \(x\) is the coordinate within the group. Compared with \(T\), the first row of \(T_2\) loses \(0\) and \(11\), the second row loses \(9\), and the fourth row loses \(2\); the new symbol \(4k\) lies in rows \(1\) and \(4\), and \(4k+1\) lies in rows \(1\) and \(2\). Only the first row may have its intersection number reduced to \(1\) because a single column contains two deleted symbols.

The column supports of symbols \(0\) and \(11\) in \(T\) are respectively
\begin{center}
$\{(0,0),(1,0),(2,0)\},\
\{(0,2),(1,5\bmod k),(2,8\bmod k)\}$.
\end{center}
For the two to intersect requires \(k\mid 5\), or \(k\mid 8\). Since \(k\ge 5\) and \(k\ne 6\), this is possible only for \(k=5\) or \(8\). For $k=5$, one can check that the array given in Example 1, obtained by the above construction, is indeed a $(4\times 15,20)$-NTA. However, for $k=8$,
note that the set of elements in the $17$-th column of $T_2$ is $\{18,11,0,25\}$, and
$0,11,25$ are not in the first row, so the first row and the $17$-th column intersect in exactly one element. This does not satisfy the row-column intersection condition. So for this case, we reconstruct $T_2$ as follows: obtain $T_2$ from $T_1$ by replacing the symbol $5$ in cell $(2,2)$ and the symbol $7$ in cell $(1,13)$ with a new symbol $33$. Then one can easily check that $T_2$ is a $(4\times 24,34)$-NTA. For the remaining values of $k$, the two columns $3$ and $(5\bmod k)+k+1$ have exactly one common symbol, namely $11$, so the replacement preserves the at most one
intersection property for pairs of columns. Again, the new symbol $4k+1$ occurs twice, the symbols $9$ and $11$ occur twice, together with the previous symbols $0,2,4k$ occur twice, and all other symbols occur three times. Hence $T_2$ is near equireplicate. The row-row and row-column intersection
conditions are preserved by the choice of the two cells. Therefore, $T_2$ is a $(4\times3k,4k+2)$-NTA.

(2) Let $T'=(t'_{i,j})$ be obtained from $T$ by first replacing the symbol $0$ in cell $(1,1)$ with the symbol
$4k$ and the symbol $10$ in cell $(3,3)$ with the symbol $4k+1$, and  then adding a column $3k+1$.
The entries in the new column are chosen so that the row intersection numbers
become the two required values. In particular, choose four
distinct symbols $0,4k,10,4k+1$ and put
\begin{center}
$t'_{1,3k+1}=0,\
t'_{2,3k+1}=4k,\
t'_{3,3k+1}=10,\
t'_{4,3k+1}=4k+1$.
\end{center}
For the present construction, it can be readily verified that $T'$ is near equireplicate and binary.  Let $\mathcal R_i'$, $1\leq i\leq 4$, denote the set of elements in the $i$-th row, and $\mathcal C_j$, $1\leq j\leq 3k+1$, denote the set of elements in the $j$-th column. Then
\begin{center}
$\mathcal R_1'=\{4a+l\mid0\leq a<k,\ l=0,2,3\}\cup\{4k\}$,\ \
$\mathcal R_2'=\{4a+l\mid0\leq a<k,\ l=0,1,3\}\cup\{4k\}$, \ \
$\mathcal R_3'=\{4a+l\mid0\leq a<k,\ l=0,1,2\}\cup\{4k+1\}$,\ \
$\mathcal R_4'=\{4a+l\mid0\leq a<k,\ l=1,2,3\}\cup\{4k+1\}$.
\end{center}
It follows that
\begin{center}
$|\mathcal R_1'\cap\mathcal R_2'|
=
|\mathcal R_3'\cap\mathcal R_4'|
=2k+1$,
\end{center}
and
\begin{center}
$|\mathcal R_1'\cap\mathcal R_3'|
=|\mathcal R_1'\cap\mathcal R_4'|
=|\mathcal R_2'\cap\mathcal R_3'|
=|\mathcal R_2'\cap\mathcal R_4'|
=2k$.
\end{center}
By construction and the column-column intersection property of $T$, any two of the first $3k$ columns intersect in either $0$ or $1$ common element. Hence, it remains only to examine the intersection of the last column with any one of the first $3k$ columns. Clearly, $\mathcal{C}_{3k+1}=\{0,4k,10,4k+1\}$, $4k$ occurs in the first column, $4k+1$ in the third, $0$ in columns $k+1$ and $2k+1$, and $10$ in columns $k+5$ and $(6 \pmod k)+2k+1$. Note that $k\geq 5$ and $k\neq 6$. Therefore, no two elements of $\mathcal{C}_{3k+1}$ occur in the same column. Consequently, the last column intersects any one of the first $3k$ columns in either $0$ or $1$ common element. This shows that any two columns of $T'$ have $0$ or $1$ common symbol.

It remains to verify the row-column intersection numbers. Every row and every column of \(T\) have exactly three common symbols. By construction, it is straightforward to check that $|\mathcal R_1'\cap \mathcal C_3|=|\mathcal R_2'\cap \mathcal C_{3k+1}|=|\mathcal R_3'\cap \mathcal C_1|=|\mathcal R_4'\cap \mathcal C_{3k+1}|=2$.
Apart from these exceptions, every other pair consisting of a row and a column intersects in exactly three elements.

Therefore, $T'$ is a $(4\times(3k+1),4k+2)$-NTA.
This completes the proof.
\end{proof}

To illustrate the proof of Lemma \ref{direct-Con}, we present the following example.
\begin{exam}\label{smallexample1}
There exist a $(4\times 15,20+\eta)$-NTA for $\eta\in\{0,1,2\}$ and a $(4\times 16,22)$-NTA.
\end{exam}
\begin{proof}
Following the construction described in Lemma \ref{direct-Con}, we present the required near triple arrays below.

\vskip 6pt
\footnotesize{}
\mbox{}\hspace{0.002in}
$T=$ \begin{tabular}{|c c c c c | c c c c c | c c c c c |}
    \hline
    0 & 4 & 8 & 12 & 16 & 11 & 15 & 19 & 3 & 7 & 6 & 10 & 14 & 18 & 2 \\
    1 & 5 & 9 & 13 & 17 & 0 & 4 & 8 & 12 & 16 & 19 & 3 & 7 & 11 & 15  \\
    2 & 6 & 10 & 14 & 18 & 17 & 1 & 5 & 9 & 13 & 0 & 4 & 8 & 12 & 16  \\
    3 & 7 & 11 & 15 & 19 & 14 & 18 & 2 & 6 & 10 & 13 & 17 & 1 & 5 & 9 \\
    \hline
    \end{tabular}

\vskip 6pt
$T_1=$    \begin{tabular}{|c c c c c | c c c c c | c c c c c |}
    \hline
    \textcolor{red}{20} & 4 & 8 & 12 & 16 & 11 & 15 & 19 & 3 & 7 & 6 & 10 & 14 & 18 & 2 \\
    1 & 5 & 9 & 13 & 17 & 0 & 4 & 8 & 12 & 16 & 19 & 3 & 7 & 11 & 15  \\
    2 & 6 & 10 & 14 & 18 & 17 & 1 & 5 & 9 & 13 & 0 & 4 & 8 & 12 & 16  \\
    3 & 7 & 11 & 15 & 19 & 14 & 18 & \textcolor{red}{20} & 6 & 10 & 13 & 17 & 1 & 5 & 9 \\
    \hline
    \end{tabular}

\vskip 6pt
$T_2=$
    \begin{tabular}{|c c c c c | c c c c c | c c c c c |}
    \hline
\textcolor{red}{20} & 4 & 8 & 12 & 16 & \textcolor{blue}{21} & 15 & 19 & 3 & 7 & 6 & 10 & 14 & 18 & 2 \\
1 & 5 & \textcolor{blue}{21} & 13 & 17 & 0 & 4 & 8 & 12 & 16 & 19 & 3 & 7 & 11 & 15  \\
 2 & 6 & 10 & 14 & 18 & 17 & 1 & 5 & 9 & 13 & 0 & 4 & 8 & 12 & 16  \\
 3 & 7 & 11 & 15 & 19 & 14 & 18 & \textcolor{red}{20} & 6 & 10 & 13 & 17 & 1 & 5 & 9 \\
    \hline
    \end{tabular}

\vskip 6pt
$T'=$
    \begin{tabular}{|c c c c c | c c c c c | c c c c c | c |}
    \hline
    \textcolor{red}{20} & 4 & 8 & 12 & 16 & 11 & 15 & 19 & 3 & 7 & 6 & 10 & 14 & 18 & 2 & \textcolor{red}{0} \\
    1 & 5 & 9 & 13 & 17 & 0 & 4 & 8 & 12 & 16 & 19 & 3 & 7 & 11 & 15 & \textcolor{blue}{20}  \\
    2 & 6 & \textcolor{red}{21} & 14 & 18 & 17 & 1 & 5 & 9 & 13 & 0 & 4 & 8 & 12 & 16 & \textcolor{red}{10} \\
    3 & 7 & 11 & 15 & 19 & 14 & 18 & 2 & 6 & 10 & 13 & 17 & 1 & 5 & 9 & \textcolor{blue}{21} \\
    \hline
    \end{tabular}

 \normalsize{}
\vskip 6pt
\noindent
It is straightforward to verify that $T$, $T_1$ and $T_2$ are the $(4\times 15,20+\eta)$-NTAs for $\eta\in\{0,1,2\}$, respectively, and $T'$ is a $(4\times 16,22)$-NTA.
\end{proof}

When $k=6$, the $(4\times 18,24)$-NTA cannot be constructed via Lemma \ref{direct-Con}. By Lemma \ref{Nec-par}, we readily see that $\lambda_{cc}\in[0,1]$.
However, as shown in the proof of Lemma \ref{direct-Con}, we can list a
$4\times 18$ row-column design on $24$ symbols as follows.

\vskip 6pt
\begin{tabular}{|cccccc | cccccc | cccccc |}\hline
\textcolor{red}{0} & \textcolor{red}{4} & \textcolor{red}{8} & \textcolor{red}{12} & \textcolor{red}{16} & \textcolor{red}{20} & 15 & 19 & 23 & 3 & 7 & 11 & 10& 14 & 18 & 22 & 2 &6 \\
1 & 5 & 9 & 13 & 17 & 21 & 0 & 4 & 8 & 12 & 16 & 20 & \textcolor{red}{3}& \textcolor{red}{7}& \textcolor{red}{11}& \textcolor{red}{15}& \textcolor{red}{19}& \textcolor{red}{23}\\
2 & 6 & 10 & 14 & 18 & 22 & 21 & 1 & 5 & 9 & 13 & 17 & \textcolor{red}{0} & \textcolor{red}{4} & \textcolor{red}{8} & \textcolor{red}{12} & \textcolor{red}{16} & \textcolor{red}{20} \\
\textcolor{red}{3} & \textcolor{red}{7} & \textcolor{red}{11} & \textcolor{red}{15} & \textcolor{red}{19} & \textcolor{red}{23} & 18 & 22 & 2 & 6 & 10 & 14 & 17 & 21 & 1 & 5 & 9 & 13 \\
\hline
\end{tabular}

\vskip 6pt
\noindent
It is easy to verify that any two rows intersect in $12$ elements, and that every row and every column intersect in $3$ elements, whereas any two columns intersect in $0,1$, or $2$ elements. For example, the first and second columns intersect in $0$ elements, the first and $7$-th columns intersect in $1$ element, and the first and $13$-th columns intersect in $2$ elements. Thus, it is not a near triple array.
For this reason, we will adjust its elements so that it satisfies the NTA conditions, thereby obtaining a $(4 \times 18, 24)$-NTA.

\begin{exam}\label{smallexample2}
There exist a $(4 \times 18, 24+\eta)$-NTA for $\eta\in\{0,1,2\}$ and a $(4\times 19, 26)$-NTA.
\end{exam}
\begin{proof}
Let

\footnotesize{}
\vskip 6pt
\noindent
\mbox{}\hspace{0.002in}
$A=$
\begin{tabular}{|cccccc | cccccc | cccccc |}
    \hline
    0 & 4 & 8 & 12 & 16 & 20 & 15 & 19 & 23 & 3 & 7 & 11 & \textcolor{blue}{6} & \textcolor{blue}{10} & \textcolor{blue}{14} & \textcolor{blue}{18} & \textcolor{blue}{22} & \textcolor{blue}{2} \\
    1 & 5 & 9 & 13 & 17 & 21 & 0 & 4 & 8 & 12 & 16 & 20 & \textcolor{blue}{23} & \textcolor{blue}{3} & \textcolor{blue}{7} & \textcolor{blue}{11} & \textcolor{blue}{15} & \textcolor{blue}{19} \\
    2 & 6 & 10 & 14 & 18 & 22 & 21 & 1 & 5 & 9 & 13 & 17 & 0 & 4 & 8 & 12 & 16 & 20 \\
    3 & 7 & 11 & 15 & 19 & 23 & 18 & 22 & 2 & 6 & 10 & 14 & 17 & 21 & 1 & 5 & 9 & 13 \\
    \hline
    \end{tabular}

\vskip 6pt
\noindent
$A_1=$
    \begin{tabular}{|cccccc | cccccc | cccccc |}
    \hline
    24 & 4 & 8 & 12 & 16 & 20 & 15 & 19 & 23 & 3 & 7 & 11 & 6 & 10 & 14 & 18 & 22 & 2 \\
    1 & 5 & 9 & 13 & 17 & 21 & 0 & 4 & 8 & 12 & 16 & 20 & 23 & 3 & 7 & 11 & 15 & 19 \\
    2 & 6 & 10 & 14 & 18 & 22 & 21 & 1 & 5 & 9 & 13 & 17 & 0 & 4 & 8 & 12 & 16 & 20 \\
    3 & 7 & 11 & 15 & 19 & 23 & 18 & 22 & 24 & 6 & 10 & 14 & 17 & 21 & 1 & 5 & 9 & 13 \\
    \hline
    \end{tabular}

\vskip 6pt
\noindent
$A_2=$       \begin{tabular}{|cccccc | cccccc | cccccc |}
    \hline
    24 & 4 & 8 & 12 & 16 & 20 & 15 & 19 & 23 & 3 & 7 & 25 & 6 & 10 & 14 & 18 & 22 & 2 \\
    1 & 5 & 25 & 13 & 17 & 21 & 0 & 4 & 8 & 12 & 16 & 20 & 23 & 3 & 7 & 11 & 15 & 19 \\
    2 & 6 & 10 & 14 & 18 & 22 & 21 & 1 & 5 & 9 & 13 & 17 & 0 & 4 & 8 & 12 & 16 & 20 \\
    3 & 7 & 11 & 15 & 19 & 23 & 18 & 22 & 24 & 6 & 10 & 14 & 17 & 21 & 1 & 5 & 9 & 13 \\
    \hline
    \end{tabular}

\vskip 6pt
\noindent
$A'=$ \begin{tabular}{|cccccc | cccccc | cccccc | c |}
    \hline
    24 & 4 & 8 & 12 & 16 & 20 & 15 & 19 & 23 & 3 & 7 & 11 & 6 & 10 & 14 & 18 & 22 & 2 & 0 \\
    1 & 5 & 9 & 13 & 17 & 21 & 0 & 4 & 8 & 12 & 16 & 20 & 23 & 3 & 7 & 11 & 15 & 19 & 24 \\
    2 & 6 & 25 & 14 & 18 & 22 & 21 & 1 & 5 & 9 & 13 & 17 & 0 & 4 & 8 & 12 & 16 & 20 & 10 \\
    3 & 7 & 11 & 15 & 19 & 23 & 18 & 22 & 2 & 6 & 10 & 14 & 17 & 21 & 1 & 5 & 9 & 13 & 25 \\
    \hline
    \end{tabular}

 \normalsize{}
\vskip 6pt
\noindent
It is straightforward to verify that $A$, $A_1$ and $A_2$ are $(4\times 18,24+\eta)$-NTAs for $\eta\in\{0,1,2\}$ respectively, and $A'$ is a $(4\times 19,26)$-NTA.
\end{proof}

The following lemma directly generalizes Lemma 5.12 in \cite{Gordeev2026}, by which it was clearly inspired. We then extend the result of Lemma \ref{direct-Con} to the general case, with the conclusion stated below.

\begin{lem}
There exists an $(m \times (m-1)k, mk)$-NTA for any $m\geq 3$ and $k\geq 1+(m-1)(m-2)$.
\end{lem}
\begin{proof}
Let $T=(t_{i,j})$, $1\leq i\leq m$, $1\leq j\leq (m-1)k$, be an $m\times(m-1)k$ row-column design on the point set
$X=\mathbb Z_{mk}=\{0,1,\ldots,mk-1\}$.
For $0\leq a<k$ and $0\leq b<m$, put the symbol $ma+b$ in the following $m-1$ cells: 
\begin{center}
$\left(\left((b+p) \bmod m\right)+1,\ \left((a+pb) \bmod k\right)+pk+1\right), \ 0 \le p \le m-2, $
\end{center}
Thus, for each symbol $ma+b$, its $m-1$ occurrences lie in the $m-1$ column groups
\begin{center}
$C_0=\{1,2,\ldots,k\},\ C_1=\{k+1,k+2,\ldots,2k\},$
$C_2=\{2k+1,2k+2,\ldots,3k\}$,\\\ $\ldots,\ C_{m-2}=\{(m-2)k+1,(m-2)k+2,\ldots,(m-1)k\}$.
\end{center}

Clearly, every symbol occurs exactly once in each of the $m-1$
column groups, and hence exactly $m-1$ times in $T$. Therefore, $T$
is equireplicate and each symbol appears at most once in each column.
Let $\mathcal R_i$, $1\leq i\leq m$, denote the set of elements in the $i$-th row.
Then we have
\begin{center}
$\mathcal R_i=\{ma+l\mid0\leq a<k,\ l\in\mathbb{Z}_{m}\setminus\{i \mod m\}\}$.
\end{center}
It follows that $|\mathcal R_i|=(m-1)k$ for each $1\leq i\leq m$. Thus, each symbol appears at most once in each row, and then $T$ is binary. It is straightforward to verify that $|\mathcal R_{i_1}\cap \mathcal R_{i_2}|=(m-2)k$ for any $1\leq i_1\neq i_2\leq m$.

We now verify the intersections between columns. Two distinct columns belonging to the same column group cannot contain a common symbol, because every symbol occurs exactly once in each column group. To facilitate computations modulo $k$, identify the columns in each group with $\mathbb Z_k$. Then the $m-1$ columns containing $ma+b$ are
\begin{center}
$a,\ \ (a+b)\pmod k,\ \ (a+2b)\pmod k, \ \,\ldots, \ \ (a+(m-2)b)\pmod k.$
\end{center}
Let $C_{p,x}$ denote the column corresponding to $x\in\mathbb Z_k$
in the $p$-th column group, where $p\in\{0,1,2,\ldots,m-2\}$. Then, we have
\begin{center}
$ma+b\in C_{p,x}
\quad\Longleftrightarrow\quad
a+pb\equiv x\pmod k$.
\end{center}
Suppose that two distinct symbols $ma+b$ and $ma'+b'$ belong to both $C_{p,x}$ and $C_{q,y}$, where $p\neq q$. If $b=b'$, then from $a + pb \equiv a' + pb' \pmod k$ and $0 \le a, a' < k$, we obtain $a = a'$, which contradicts the distinctness of the symbols. Hence, $b \neq b'$. Then $(p-q)(b-b')\equiv 0 \pmod k$, we have $k\mid(p-q)(b-b')$. Now, $0<|p-q|\le m-2$ and $0<|b-b'|\le m-1$, so $0<|(p-q)(b-b')|\leq (m-2)(m-1)$, since $k\geq 1+(m-1)(m-2)> (m-2)(m-1)$, this is a contradiction. Thus, any two distinct columns have $0$ or $1$ common symbol.

Finally, we prove that every row and every column of \(T\) have exactly $m-1$ common symbols. Fix a symbol $\alpha=ma+b,\ 0\le a<k,\ 0\le b<m$.
By the construction, \(\alpha\) occurs in the following $m-1$ rows:
$b+1,\ [(b+1)\pmod m]+1,\ [(b+2)\pmod m]+1,\ldots,\ [(b+(m-2))\pmod m]+1$. More explicitly,
$$\begin{array}{c|c}
b & \text{rows containing }ma+b\\ \hline
0 & 1,2,3,\ldots,m-2,m-1\\
1 & 2,3,4,\ldots,m-1,m\\
2 & 3,4,5,\ldots,m,1\\
\vdots& \vdots\\
m-1 & m,1,2,\ldots,m-3,m-2
\end{array}$$
Hence, each symbol occurs in exactly $m-1$ rows and is absent from exactly one row. The unique row in which \(ma+b\) is absent is
$$\begin{cases}
m,&b=0,\\
1,&b=1,\\
2,&b=2,\\
~\vdots&~~~\vdots\\
m-1,&b=m-1,\\
\end{cases}$$
Thus, as \(b\) runs through \(0,1,2,\ldots,m-1\), the $m$ symbols \(ma+b\) are respectively absent from the $m$ different rows.
Now fix an arbitrary column \(C_{p,x}\), where \(p\in\{0,1,2,\ldots,m-2\}\) and \(x\in\mathbb Z_k\). From the construction, the symbol \(ma+b\) occurs in \(C_{p,x}\) precisely when $a+pb\equiv x\pmod k$. In fact, for the first group, that is, \(p=0\), the column coordinate is \(a\); for the second group \(p=1\), the column coordinate is $(a+b)\pmod k$; for the third group \(p=2\), the column coordinate is $(a+2b)\pmod k$, $\ldots$, for the $m-1$ group \(p=m-2\), the column coordinate is $(a+(m-2)b)\pmod k$. For each fixed \(b\in\{0,1,2,\ldots,m-1\}\), the congruence
$a+pb\equiv x\pmod k$ has a unique solution $a\equiv x-pb\pmod k$.
Therefore, the column \(C_{p,x}\) contains exactly one symbol corresponding to each value of \(b\). In other words, there are exactly $m$ symbols in \(C_{p,x}\), say
$x_b=ma_b+b,\ b=0,1,2,\ldots,m-1$, where $a_b\equiv x-pb\pmod k$.
Hence, if \(\mathcal{R}_i\) is any fixed row, exactly one of the $m$ symbols in \(C_{p,x}\) is absent from \(\mathcal{R}_i\). The other $m-1$ symbols occur in \(\mathcal{R}_i\).
Therefore, $|\mathcal{R}_i\cap C_{p,x}|=m-1$, we conclude that every row and every column have exactly $m-1$ common symbols. This proves the required row column intersection condition. It follows that, for $m\geq 3$ and $k\geq 1+(m-1)(m-2)$, $T$ is an $(m\times(m-1)k,mk)$-NTA.
\end{proof}

\subsection{Constructions for NTAs based on matching designs}

Let $v, k$ and $\lambda$ be integers such that \(v \ge k \ge 2\) and \(\lambda \ge 1\).
Let \(X\) be a finite set of elements, called \emph{points}, and let \(\mathcal{A}\) be a finite collection of subsets of \(X\), called \emph{blocks}. The pair \((X, \mathcal{A})\) is called a \((v,k,\lambda)\) \emph{balanced incomplete block design} or, simply, a \((v,k,\lambda)\)-BIBD, if the following conditions hold:

(i) \(|X| = v\).

(ii) \(|A| = k\) for all \(A \in \mathcal{A}\).

(iii) Every pair of distinct points is contained in exactly \(\lambda\) blocks.

\noindent
It is easy to calculate that a $(v,k,\lambda)$-BIBD has exactly $b=\frac{\lambda v(v-1)}{k(k-1)}$ blocks and every point occurs in exactly  $r=\frac{\lambda(v-1)}{k-1}$ blocks. Therefore, a $(v,k,\lambda)$-BIBD is also called a $(v,b,r,k,\lambda)$-BIBD and $\{v,b,r,k,\lambda\}$ is called the set of its \emph{parameters}.

The \emph{column design} of a triple array is a block design with points and blocks corresponding to, respectively, columns and symbols of the array, with a point appearing in a block when the corresponding column contains the corresponding symbol. The \emph{row design} can be defined similarly. In the column design of a triple array, all blocks have the same size and each pair of points is covered by the same number of blocks, so it is a balanced incomplete block design (BIBD).

It is worth noting that in the definition of a BIBD, the elements within a block are unordered. However, the subsequent definitions require ordered structures, yet we continue to use set notation. For two sequences
$D_1=\{x_{1},x_2,\cdots,x_k\}$ and $D_2=\{y_{1},y_2,\cdots,y_k\}$, write $D_1D_2$ for the sequence of $k$ ordered pairs $x_1y_1$, $x_2y_2$, $\cdots$, $x_ky_k$. Let $(X, \mathcal{A})$ and $(Y, \mathcal{C})$ be a $(v,b,r,k,\lambda_1)$-BIBD and a $(r,b,v,k,\lambda_2)$-BIBD, respectively, where $\mathcal{A}=\{A_1,A_2,\cdots,A_b\}$ and $\mathcal{C}=\{C_1,C_2,\cdots,C_b\}$. $(X, \mathcal{A})$ and $(Y, \mathcal{C})$ are called \emph{a pair of matching designs} if every member of $X\times Y$ occurs exactly once in $\mathcal{A}\mathcal{C}=\{A_1C_1,A_2C_2,\cdots,A_bC_b\}$, namely $X\times Y=\bigcup\limits_{i=1}^bA_iC_i$.

\begin{exam}\label{MD}
Let $X=\{1,2,3,4,5\}$, $\mathcal{A}=\{A_1,A_2,\cdots,A_{10}\}$, $Y=\{1,2,3,4,5,6\}$, $\mathcal{C}=\{C_1,C_2,\cdots,C_{10}\}$.
One can easily check that $(X,\mathcal{A})$ and $(Y,\mathcal{C})$ are a $(5,10,6,3,3)$-BIBD and a $(6,10,5,3,2)$-BIBD, respectively, and they are a pair of matching designs.

\vskip 6pt
\mbox{}\hspace{0.8in}
$\begin{array}{|cccc|ccccc|}\hline
A_1:\ 1\ \ 2\ \ 4 && C_1:\ \ 1\ \ 3\ \ 5 &&& A_1C_1:\ 11 & 23 & 45 &  \\
A_2:\ 1\ \ 3\ \ 4 && C_2:\ \ 2\ \ 1\ \ 6 &&& A_2C_2:\ 12 & 31 & 46 &  \\
A_3:\ 1\ \ 2\ \ 3 && C_3:\ \ 3\ \ 4\ \ 6 &&& A_3C_3:\ 13 & 24 & 36 &  \\
A_4:\ 1\ \ 2\ \ 5 && C_4:\ \ 4\ \ 1\ \ 2 &&& A_4C_4:\ 14 & 21 & 52 &  \\
A_5:\ 1\ \ 3\ \ 5 && C_5:\ \ 5\ \ 2\ \ 3 &&& A_5C_5:\ 15 & 32 & 53 &  \\
A_6:\ 1\ \ 4\ \ 5 && C_6:\ \ 6\ \ 4\ \ 5 &&& A_6C_6:\ 16 & 44 & 55 & \\
A_7:\ 2\ \ 4\ \ 5 && C_7:\ \ 2\ \ 3\ \ 6 &&& A_7C_7:\ 22 & 43 & 56 & \\
A_8:\ 2\ \ 3\ \ 4 && C_8:\ \ 5\ \ 4\ \ 2 &&& A_8C_8:\ 25 & 34 & 42 &\\
A_9:\ 2\ \ 3\ \ 5 && C_9:\ \ 6\ \ 5\ \ 1 &&& A_9C_9:\ 26 & 35 & 51 & \\
A_{10}: 3\ \ 4\ \ 5 && C_{10}:\ \ 3\ \ 1\ \ 4 &&& A_{10}C_{10}:\ 33 & 41 & 54 & \\\hline
\end{array}$
\end{exam}

Matching designs have been widely applied in the construction of double arrays. Relevant detailed approaches are presented in McSorley \cite{McSorley37}.

\begin{con}[\!\!\cite{McSorley37}, Lemma 2.3]\label{MD-DA}
There exists a pair of matching designs, a $(v,b,r,k,\lambda_1)$-BIBD and a $(r,b,v,k,\lambda_2)$-BIBD, if and only if there exists a $(v\times r,b)$-DA, i.e., a
$(b,k,\lambda_{1},\lambda_{2}:v\times r)$-DA.
\end{con}
\begin{proof}
Suppose that $(X,\mathcal{A})$ and $(Y,\mathcal{D})$ are a $(v,b,r,k,\lambda_1)$-BIBD and an $(r,b,v,k,\lambda_2)$-BIBD, respectively, and they are a pair of matching designs.
Let $X=\{1,2,\ldots,v\}$, $\mathcal{A}=\{A_1,A_2,\cdots,A_{b}\}$, $Y=\{1,2,\ldots,r\}$, $\mathcal{D}=\{D_1,D_2,\cdots,D_{b}\}$, where $A_i=\{a_{i1},a_{i2},\ldots,a_{ik}\}$ and $D_i=\{d_{i1},d_{i2},\ldots,d_{ik}\}$, $i=1,2,\ldots,b$.

Define a $v\times r$ array $T=(t_{a_{ij},d_{ij}})$, $i=1,2,\ldots,b$ and $j=1,2,\ldots,k$, where $t_{a_{ij},d_{ij}}=i$.
Each symbol $i\in\{1,2,\ldots,b\}$ occurs at most once in each
row or column, as the elements of each block of $\mathcal{A}$ or $\mathcal{D}$ are distinct; each symbol $i\in\{1,2,\ldots,b\}$ appears exactly $k$ times in the array since $j$ has $k$ distinct choices, that is, $j\in\{1,2,\ldots,k\}$.
In addition, any two distinct rows contain $\lambda_1$ common symbols since $(X,\mathcal{A})$ is a $(v,b,r,k,\lambda_1)$-BIBD, and any two distinct columns contain $\lambda_2$ common symbols since $(Y,\mathcal{D})$ is a $(r,b,v,k,\lambda_2)$-BIBD. It follows that $T$ is a $(v\times r,b)$-DA.

Suppose that $T=(t_{ij})$ is a $(v\times r,b)$-DA. Let $R_i=\{t_{i1},t_{i2},\cdots,t_{ir}\}$, $i=1,2,\cdots,v$, and $C_j=\{t_{1j},t_{2j},\cdots,t_{vj}\}$, $j=1,2,\cdots,r$. For
each $l=1,2,\cdots,b$, define $A_l=\{x\mid l\in R_x\}$ and $D_l=\{x\mid l\in C_x\}$. It is easy to verify that $(X,\mathcal{A})$ and $(Y,\mathcal{D})$ are a $(v,b,r,k,\lambda_1)$-BIBD and a $(r,b,v,k,\lambda_2)$-BIBD, respectively, where $\mathcal{A}=\{A_1,A_2,\cdots,A_b\}$ and
$\mathcal{D}=\{D_1,D_2,\cdots,D_b\}$.
This completes the proof.
\end{proof}

\textbf{Remark 1}~  As shown in Construction \ref{MD-DA}, matching designs are applicable for constructing double arrays. However, such designs do not always produce triple arrays.

The blocks $A_i$ and $D_i$ $(i=1,2,\ldots,6)$ are given below. Let $X=\{1,2,3\}$, $\mathcal{A}=\{A_1,A_2,\cdots,A_{6}\}$, $Y=\{1,2,3,4\}$ and $\mathcal{D}=\{D_1,D_2,\cdots,D_{6}\}$. It can be readily seen that
$(X,\mathcal{A})$ and $(Y,\mathcal{D})$ are a $(3,6,4,2,2)$-BIBD and a $(4,6,3,2,1)$-BIBD, respectively, and they are a pair of matching designs.

\vskip 6pt
\mbox{}\hspace{1.0in}
$\begin{array}{|cccc|cccc|}\hline
A_1:\ 1\ \ 2\  && D_1:\ \ 1\ \ 2\  &&& A_1D_1:\ 11 & 22 &   \\
A_2:\ 1\ \ 3\  && D_2:\ \ 2\ \ 3\  &&& A_2D_2:\ 12 & 33 &   \\
A_3:\ 1\ \ 2\  && D_3:\ \ 3\ \ 1\  &&& A_3D_3:\ 13 & 21 &   \\
A_4:\ 1\ \ 3\  && D_4:\ \ 4\ \ 2\  &&& A_4D_4:\ 14 & 32 &  \\
A_5:\ 2\ \ 3\  && D_5:\ \ 3\ \ 4\  &&& A_5D_5:\ 23 & 34 &   \\
A_6:\ 2\ \ 3\  && D_6:\ \ 4\ \ 1\  &&& A_6D_6:\ 24 & 31 &   \\\hline
\end{array}$

\vskip 6pt

\noindent
By Construction \ref{MD-DA}, we obtain a $(3\times 4,6)$-DA, $T$, which is listed below.

\begin{center}
$T=\begin{array}{|cccc|}\hline
  1 & 2 & 3 & 4  \\
  3 & 1 & 5 & 6  \\
  6 & 4 & 2 & 5 \\\hline
\end{array}$
\end{center}
Not every double array is a triple array. In fact, an exhaustive search shows that there is no $(3\times 4,6)$ triple array \cite{McSorley2005}. Thus, the existence of matching designs is only a necessary condition for the existence of triple arrays, but not a sufficient one.

\begin{con}\label{MD-TA}
Let $(X, \mathcal{A})$ be a $(v,b,r,k,\lambda_1)$-BIBD, where $X = \{1,2,\dots,v\}$ and $\mathcal{A} = \{A_1,A_2,\dots,A_b\}$, and $(Y, \mathcal{D})$ be an  $(r,b,v,k,\lambda_2)$-BIBD, where $Y = \{1,2,\dots,r\}$ and $\mathcal{D} = \{D_1,D_2,\dots,D_b\}$, and they are a pair of matching designs. Define
$R_i=\{x\mid i\in A_x\}$ for any $i\in X$ and $C_j=\{x\mid j\in D_x\}$ for any $j\in Y$. If $\lvert R_i \cap C_j \rvert =\lambda_{3}$ for any $i\in X$ and $j\in Y$, then there exists a $(v\times r,b)$-TA, namely a $(b,k,\lambda_{1},\lambda_{2},\lambda_{3}:v\times r)$-TA.
\end{con}
\begin{proof}
Let $T$  be defined as the proof of Construction \ref{MD-DA}, then $T$ is a $(v\times r,b)$-DA. Since $\lvert R_i \cap C_j \rvert =\lambda_{3}$ for any $i\in X$ and $j\in Y$, every row and every column contain precisely $\lambda_{3}$ common symbols. This completes the proof.
\end{proof}

The existence of a $(5\times 6, 10)$-TA
was already given by Agrawal in 1966 \cite{Agrawal1966}. Below, to illustrate Construction \ref{MD-TA}, we also provide an example for this small parameter.
\begin{exam}
There exists a $(5\times 6, 10)$-TA.
\end{exam}
\begin{proof}
By Construction \ref{MD-TA} and Example \ref{MD}, a $(5\times 6, 10)$-TA is obtained and listed below.
\begin{center}
$T=\begin{array}{|cccccc|}\hline
  1 & 2 & 3 & 4 & 5 & 6 \\
  4 & 7 & 1 & 3 & 8 & 9 \\
  2 & 5 & 10 & 8 & 9 & 3 \\
  10 & 8 & 7 & 6 & 1 & 2 \\
  9 & 4 & 5 & 10 & 6 & 7 \\\hline
\end{array}$
\end{center}
\end{proof}

The definition of a near triple array is derived by relaxing the three intersection conditions of triple array. Accordingly, for certain arrays where the parameters
$e$, $\lambda_{rr}$, $\lambda_{cc}$ and $\lambda_{rc}$ may form a binary set (i.e., take two consecutive values), the row and column designs are not balanced incomplete block designs any more, but maximally balanced maximally uniform designs, which were studied by Bofill and Torras \cite{Bofill2004}.

\emph{Maximally balanced maximally uniform designs} (MBMUDs) are block designs in which the block sizes are either constant or each equal to one of two adjacent integers, and the same condition holds for occurrence numbers of symbols (also called \emph{replication numbers}), and for covering numbers of pairs of symbols (also called \emph{concurrences}). We use the notation $(v,b,\{r^-,r^+\},\{k^-,k^+\},\{\lambda_{cc}^-,\lambda_{cc}^+\})$-MBMUD or $(v,b,r,k,\lambda_{cc})$-MBMUD for short, where $v$ is the number of points, $b$ the number of blocks, the replication number of each point is either $r^-$ or $r^+$, the block size is either $k^-$ or $k^+$, and the number of occurrences of each pair of points is $\lambda_{cc}^-$ or $\lambda_{cc}^+$.  If the replication number of each point in an MBMUD is constant, say $r$, then it is called \emph{regular}. For our further study, the MBMUDs adopted are all regular. Similarly, we define \emph{matching designs for maximally balanced maximally uniform designs} (MBMUDs), whose definition is analogous to that for BIBDs. Let $(X, \mathcal{A})$ and $(Y, \mathcal{C})$ be a $(v,b,r,\{k^-,k^+\},\{\lambda_1^-,\lambda_1^+\})$-MBMUD and a $(r,b,v,\{k^-,k^+\},\{\lambda_2^-,\lambda_2^+\})$-MBMUD, respectively, where $\mathcal{A}=\{A_1,A_2,\cdots,A_b\}$ and $\mathcal{C}=\{C_1,C_2,\cdots,C_b\}$. $(X, \mathcal{A})$ and $(Y, \mathcal{C})$ are called \emph{a pair of matching designs} if every member of $X\times Y$ occurs exactly once in $\mathcal{A}\mathcal{C}=\{A_1C_1,A_2C_2,\cdots,A_bC_b\}$, i.e., $X\times Y=\bigcup\limits_{i=1}^bA_iC_i$. Note that, the numbers of block size $k^-$ of these two MBMUDs are the same. We give an example of a pair of matching designs for maximally balanced maximally uniform designs in Example \ref{MD-MBMUD} below.

\begin{exam}
$(X, \mathcal{A})$ is a $(13,16,4,\{3,4\},\{0,1\})$-MBMUD, where

\mbox{}\hspace{1.4in}
$X = \{1,2,3,4,5,6,7,8,9,10,11,12,13\}$,

$\mathcal{A} = \{\{1,2,3,4\}, \{1,5,6,7\}, \{1,8,9,10\}, \{1,11,12,13\}, \{2,5,8\}, \{3,6,9\}, \{4,7,10\}$,

\mbox{}\hspace{0.38in}
$\{2,9,11\},~ \{3,10,12\}, ~\{4,8,13\}, ~~\{2,6,12\},~ \{3,7,13\},~ \{4,5,11\},~ \{5,9,12\}, $

\mbox{}\hspace{0.43in}$\{6,10,13\}, \{7,8,11\}\}.$
\end{exam}

It follows that the column design of an NTA is an MBMUD with fixed replication numbers, and the row design has an analogous structure. If one or both of the parameters $e$ and $\lambda_{cc}$ for an NTA are integers, the corresponding column design belongs to other thoroughly investigated design classes.

Namely, if $e$ is an integer, then column design of an NTA has constant block sizes and replication numbers, and covering numbers of pairs $\lambda_{cc}^{-}$ or $\lambda_{cc}^{+}$. Such block designs are known as \emph{regular graph designs}, which were originally introduced by John and Mitchell\cite{John1977}. If $\lambda_{cc}$ is an integer, then the column design of an NTA has constant replication numbers and covering numbers, while its block sizes are either $e^{-}$ or $e^{+}$. This yields a $(c, \{e^{-}, e^{+}\}, \lambda_{cc})$-pairwise balanced design (PBD). For a detailed introduction to pairwise balanced designs, see Part IV of the \emph{Handbook of Combinatorial Designs}\cite{Handbook}. Finally, if both $e$ and $\lambda_{cc}$ are integers, as happens for triple array, then each block of column design of the NTA contains exactly $e$ symbols, and any pair of symbols is covered by $\lambda_{cc}$ blocks. Hence, the column design of the NTA is a $(c, e, \lambda_{cc})$-BIBD. The same conclusion holds for the row design.

The proof of the following theorem is analogous to that of Construction \ref{MD-TA}, so it is omitted here.

\begin{thm}\label{CNTA}
Let $(X, \mathcal{A})$ be a $(v,b,r,\{k^-,k^+\},\{\lambda_1^-,\lambda_1^+\})$-MBMUD, where $X = \{1,2,\\\dots,v\}$ and $\mathcal{A} = \{A_1, A_2,\dots, A_b\}$, and $(Y, \mathcal{B})$ be a $(r,b,v,\{k^-,k^+\},\{\lambda_2^-,\lambda_2^+\})$-MBMUD, where $Y= \{1,2,\dots,r\}$ and $\mathcal{B} = \{B_1, B_2,\dots, B_b\}$, and they are a pair of matching designs. Define $R_i=\{x\mid i\in A_x\}$ for any $i\in X$ and $C_j=\{x\mid j\in B_x\}$ for any $j\in Y$.
If $\lvert R_i \cap C_j \rvert \in\{\lambda_3^-,\lambda_3^+\}$ for any $i\in X$ and $j\in Y$, then there exists a $(v\times r,b)$-NTA, i.e., a $(b,\{k^-,k^+\},\{\lambda_1^-,\lambda_1^+\},\{\lambda_2^-,\lambda_2^+\},\{\lambda_3^-,\lambda_3^+\}:v\times r)$-NTA.
\end{thm}

\begin{exam}\label{MD-MBMUD}
The blocks $A_i$ and $B_i(i=1,2,\ldots,16)$ are given below. Let $X=\{1,2,3,4\}$, $\mathcal{A}=\{A_1,A_2,\ldots,A_{16}\}$, $Y=\{1,2,\ldots,13\}$ and $\mathcal{B}=\{B_1,B_2,\ldots,B_{16}\}$. It can be readily seen that
$(X,\mathcal{A})$ and $(Y,\mathcal{B})$ are a $(4,16,13,\{3,4\},10)$-MBMUD and a $(13,16,4,\{3,4\},\{0,\\1\})$-MBMUD, respectively.

\mbox{}\hspace{0.4in}
$\begin{array}{|cccc|cccccc|}\hline
A_1: 1\ 2\ 3\ 4 && B_1: 4\ 3\ 1\ 2 &&& A_1B_1: 1,4 & 2,3 & 3,1 & 4,2 & \\
A_2: 1\ 2\ 3\ 4 && B_2: 1\ 5\ 7\ 6 &&& A_2B_2: 1,1 & 2,5 & 3,7 & 4,6 &  \\
A_3: 1\ 2\ 3\ 4 && B_3: 10\ 1\ 8\ 9 &&& A_3B_3: 1,10 & 2,1 & 3,8 & 4,9 &  \\
A_4: 1\ 2\ 3\ 4 && B_4: 11\ 13\ 12\ 1 &&& A_4B_4: 1,11 & 2,13 & 3,12 & 4,1 &\\
A_5: 1\ 2\ 3   && B_5: 2\ 8\ 5 &&& A_5B_5: 1,2 & 2,8 & 3,5 & & \\
A_6: 1\ 2\ 3    && B_6: 9\ 6\ 3 &&& A_6B_6: 1,9 & 2,6 & 3,3 & &\\
A_7: 1\ 2\ 3    && B_7: 7\ 10\ 4 &&& A_7B_7: 1,7 & 2,10 & 3,4 & &\\
A_8: 2\ 3\ 4    && B_8: 9\ 2\ 11 &&& A_8B_8: 2,9 & 3,2 & 4,11 & &\\
A_9: 2\ 3\ 4    && B_9: 12\ 10\ 3 &&& A_9B_9: 2,12 & 3,10 & 4,3 &&\\
A_{10}: 2\ 3\ 4 && B_{10}: 4\ 13\ 8 &&& A_{10}B_{10}: 2,4 & 3,13 & 4,8 && \\
A_{11}: 1\ 2\ 4 && B_{11}: 6\ 2\ 12 &&& A_{11}B_{11}: 1,6 & 2,2 & 4,12 && \\
A_{12}: 1\ 2\ 4 && B_{12}: 3\ 7\ 13 &&& A_{12}B_{12}: 1,3 & 2,7& 4,13 & &\\
A_{13}: 1\ 2\ 4 && B_{13}: 5\ 11\ 4 &&& A_{13}B_{13}: 1,5& 2,11 & 4,4 && \\
A_{14}: 1\ 3\ 4 && B_{14}: 12\ 9\ 5 &&& A_{14}B_{14}: 1,12 & 3,9 & 4,5 && \\
A_{15}: 1\ 3\ 4 && B_{15}: 13\ 6\ 10 &&& A_{15}B_{15}: 1,13 & 3,6 & 4,10 && \\
A_{16}: 1\ 3\ 4 && B_{16}: 8\ 11\ 7 &&& A_{16}B_{16}: 1,8 & 3,11 & 4,7 & &\\\hline
\end{array}$

\vskip 6pt
\noindent
It is easy to verify that these two MBMUDs satisfy the conditions of Theorem \ref{CNTA}, i.e., they are a pair of matching designs, and $|R_i\cap C_j|\in\{3,4\}$ for any $i\in X$ and $j\in Y$. Therefore, by Theorem \ref{CNTA}, a $(4\times13, 16)$-NTA (or a $(16,\{3,4\},10,\{0,1\},\{3,4\}:4\times 13)$-NTA) is obtained and listed below.

\begin{center}
$\begin{array}{|ccccccccccccc|}\hline
2 & 5 & 12 & 1 & 13 & 11 & 7 & 16 & 6 & 3 & 4 & 14 & 15 \\
3 & 11 & 1 & 10 & 2 & 6 & 12 & 5 & 8 & 7 & 13 & 9 & 4 \\
1 & 8 & 6 & 7 & 5 & 15 & 2 & 3 & 14 & 9 & 16 & 4 & 10 \\
4 & 1 & 9 & 13 & 14 & 2 & 16 & 10 & 3 & 15 & 8 & 11 & 12 \\\hline
\end{array}$
\end{center}
\end{exam}

\begin{lem}\label{small1-c-v}
A $(4\times c, v)$-NTA exists for any parameters $c$ and $v$ listed below.

\vskip 6pt
\centering
\begin{tabular}{|c c|c c|}
\hline
$c=13$ & \(13\leq v \leq23\) & $c=14$ & \(14\leq v \leq23\) \\
$c=15$ & \(15\leq v \leq22\) & $c=16$ & \(16\leq v \leq23\) \\
$c=17$ & \(17\leq v \leq23\) & $c=18$& \(18\leq v \leq23\) \\
$c=19$ & \(19\leq v \leq24\) & $c=20$ & \(20\leq v \leq25\) \\
$c=21$ & \(21\leq v \leq24\) & $c=22$ & \(22\leq v \leq24\) \\
$c=23$ & \(23\leq v \leq25\) & $c=24$ & \(24\leq v \leq26\) \\
$c=25$ & \(25\leq v \leq27\) &  &  \\\hline
\end{tabular}
\end{lem}
\begin{proof}
For $c \leq v \leq c+2$, a $(4 \times c, v)$-NTA exists by Lemma \ref{4nn} and Lemma \ref{4n-2-n}. For $(c,v)\in\{(15,20),(15,21),(15,22),(16,22)\}$, a $(4 \times c, v)$-NTA exists by Lemma \ref{direct-Con} or Example \ref{smallexample1}. For the other desired NTAs, we construct the matching designs for maximally balanced maximally uniform designs with the properties stated as Theorem \ref{CNTA}. To save space, we do not list these matching designs (likewise Example \ref{MD-MBMUD}) one by one, and only list the corresponding NTAs in Appendix A.
\end{proof}

\section{Proof of Theorem \ref{MainTH}}

In this section, we will provide the proof of Theorem \ref{MainTH}. The general structure of the proof of Theorem \ref{MainTH} largely follows the proof of Theorem 5.10 in \cite{Gordeev2026}. We first need to construct an $(r\times c, v)$-NTA using the following construction and results.

\begin{lem}[\!\!\cite{Gordeev2026}, Lemma 5.6]\label{T1+T2}
Let $T_1$ and $T_2$ be an $(r \times c_1, v_1)$-NTA and an $(r \times c_2, v_2)$-NTA on disjoint sets of symbols with parameters $e_1, \lambda_{rr}^1, \lambda_{cc}^1, \lambda_{rc}^1$ and $e_2, \lambda_{rr}^2, \lambda_{cc}^2, \lambda_{rc}^2$, respectively. Note that these parameters $e_i, \lambda_{rr}^i, \lambda_{cc}^i$ and $\lambda_{rc}^i$ for $i=1,2$ refer to the average values of the corresponding parameters. If the following conditions are satisfied:

1. $\max(e_1^+, e_2^+) - \min(e_1^-, e_2^-) \leq 1$,

2. $\lambda_{rr}^1 \in \mathbb{Z}$ or $\lambda_{rr}^2 \in \mathbb{Z}$,

3. $\lambda_{cc}^1, \lambda_{cc}^2 \leq 1$.

\noindent
Let $T$ be a concatenation of $T_1$ and $T_2$ (each row of $T$ is a concatenation of two corresponding rows of $T_1$ and $T_2$), then $T$ is an $(r \times (c_1 + c_2), v_1 + v_2)$-NTA.
\end{lem}

Regarding the existence of NTAs with small parameters, we have the following result, which is due to \cite{Gordeev2026}.

\begin{lem}\label{1small-c-v}
A $(4\times c, v)$-NTA exists for any parameters $4\leq c\leq 12$ and $c\leq v\leq \frac{7c}{2}$ except for $(c,v)\in\{ (4,9),(5,7),(5,10),(6,8),(7,9),(10,12),(11,13)\}$.
\end{lem}
\begin{proof}
The conclusion follows from Table B.2 in Appendix B of \cite{Gordeev2026}.
\end{proof}

\begin{lem}\label{small-c-v}
A $(4\times c, v)$-NTA exists for any parameters $c$ and $v$ listed below.

\vskip 6pt
\centering
\begin{tabular}{|c c|c c|}
\hline
$c=13$ & \(24\leq v \leq26\) & $c=14$ & \(24\leq v \leq28\) \\
$c=15$ & \(23\leq v \leq30\) & $c=16$ & \(24\leq v \leq32\) \\
$c=17$ & \(24\leq v \leq34\) & $c=18$& \(24\leq v \leq36\) \\
$c=19$ & \(25\leq v \leq38\) & $c=20$ & \(26\leq v \leq40\) \\
$c=21$ & \(25\leq v \leq42\) & $c=22$ & \(25\leq v \leq44\) \\
$c=23$ & \(26\leq v \leq46\) & $c=24$ & \(27\leq v \leq48\) \\
$c=25$ & \(28\leq v \leq50\) &  &  \\\hline
\end{tabular}
\end{lem}
\begin{proof}
In the first and second columns of the table below, all NTAs with at most $12$ columns come from Lemma \ref{1small-c-v}, and the NTAs with $13$ columns in the following table come from Lemma \ref{small1-c-v}. It is readily verified that they satisfy the conditions of Lemma \ref{T1+T2}. Accordingly, we obtain the new NTAs in the third column by applying Lemma \ref{T1+T2}.


\footnotesize{}
\begin{center}
\begin{tabular}{|cc|c|}
\hline
The known NTA & The known NTA & New NTA\\\hline
\((4\times6,\, 12)\)-NTA & \((4\times7,\, v-12)\)-NTA for \(v\in [24,26]\) & \((4\times13,\, v)\)-NTA \\
\((4\times6,\, 12)\)-NTA & \((4\times8,\, v-12)\)-NTA for \(v\in [24,28]\) & \((4\times14,\, v)\)-NTA \\
\((4\times6,\, 11)\)-NTA & \((4\times9,\, 12)\)-NTA & \((4\times15,\, 23)\)-NTA \\
\((4\times6,\, 12)\)-NTA & \((4\times9,\, v-12)\)-NTA for \(v\in [24,30]\) & \((4\times15,\, v)\)-NTA \\
\((4\times9,\, 12)\)-NTA & \((4\times7,\, v-12)\)-NTA for \(v\in [24,26]\) & \((4\times16,\, v)\)-NTA \\
\((4\times6,\, 12)\)-NTA & \((4\times10,\, v-12)\)-NTA for \(v\in [27,32]\) & \((4\times16,\, v)\)-NTA \\
\((4\times9,\, 12)\)-NTA & \((4\times8,\, v-12)\)-NTA for \(v\in [24,28]\) & \((4\times17,\, v)\)-NTA \\
\((4\times6,\, 12)\)-NTA & \((4\times 11,\, v-12)\)-NTA for \(v\in [29,34]\) & \((4\times17,\, v)\)-NTA \\
\((4\times9,\, 12)\)-NTA & \((4\times9,\, v-12)\)-NTA for \(v\in [24,30]\) & \((4\times18,\, v)\)-NTA \\
\((4\times9,\, 18)\)-NTA & \((4\times9,\, v-18)\)-NTA for \(v\in [31,36]\) & \((4\times18,\, v)\)-NTA \\
\((4\times9,\, 12)\)-NTA & \((4\times10,\, v-12)\)-NTA for \(v\in [25,32]\) & \((4\times19,\, v)\)-NTA \\
\((4\times9,\, 16)\)-NTA & \((4\times10,\, v-16)\)-NTA for \(v\in [33,36]\) & \((4\times19,\, v)\)-NTA \\
\((4\times9,\, 18)\)-NTA & \((4\times10,\, v-18)\)-NTA for \(v\in [37,38]\) & \((4\times19,\, v)\)-NTA \\
\((4\times9,\, 12)\)-NTA & \((4\times11,\, v-12)\)-NTA for \(v\in [26,32]\) & \((4\times20,\, v)\)-NTA \\
\((4\times9,\, 18)\)-NTA & \((4\times11,\, v-18)\)-NTA for \(v\in [33,40]\) & \((4\times20,\, v)\)-NTA \\
\((4\times9,\, 12)\)-NTA & \((4\times12,\, v-12)\)-NTA for \(v\in [25,28]\) & \((4\times21,\, v)\)-NTA \\
\((4\times12,\, 16)\)-NTA & \((4\times9,\, v-16)\)-NTA for \(v\in [29,34]\) & \((4\times21,\, v)\)-NTA \\
\((4\times9,\, 18)\)-NTA & \((4\times12,\, v-18)\)-NTA for \(v\in [35,42]\) & \((4\times21,\, v)\)-NTA \\
\((4\times9,\, 12)\)-NTA & \((4\times13,\, v-12)\)-NTA for \(v\in [25,35]\) & \((4\times22,\, v)\)-NTA \\\hline
\end{tabular}
\end{center}

\normalsize{}
For the conclusions in the table below, as in the proof above, the known NTAs in the first column come from Lemma \ref{1small-c-v}, while those in the second column come from Lemma \ref{small1-c-v} or the results already proved above. It is readily verified that they satisfy the conditions of Lemma \ref{T1+T2}. Accordingly, we derive the new NTAs in the third column by applying Lemma \ref{T1+T2}.

\footnotesize{}
\begin{center}
\begin{tabular}{|cc|c|}
\hline
The known NTA & The known NTA & New NTA\\\hline
\((4\times9,\, 18)\)-NTA & \((4\times13,\, v-18)\)-NTA for \(v\in [36,44]\) & \((4\times22,\, v)\)-NTA \\
\((4\times9,\, 12)\)-NTA & \((4\times14,\, v-12)\)-NTA for \(v\in [26,36]\) & \((4\times23,\, v)\)-NTA \\
\((4\times9,\, 18)\)-NTA & \((4\times14,\, v-18)\)-NTA for \(v\in [37,46]\) & \((4\times23,\, v)\)-NTA \\
\((4\times9,\, 12)\)-NTA & \((4\times15,\, v-12)\)-NTA for \(v\in [27,37]\) & \((4\times24,\, v)\)-NTA \\
\((4\times9,\, 18)\)-NTA & \((4\times15,\, v-18)\)-NTA for \(v\in [38,48]\) & \((4\times24,\, v)\)-NTA \\
\((4\times9,\, 12)\)-NTA & \((4\times16,\, v-12)\)-NTA for \(v\in [28,39]\) & \((4\times25,\, v)\)-NTA \\
\((4\times9,\, 18)\)-NTA & \((4\times16,\, v-18)\)-NTA for \(v\in [40,50]\) & \((4\times25,\, v)\)-NTA \\ \hline
\end{tabular}
\end{center}
\end{proof}

\begin{lem}[\!\cite{Gordeev2026}, Lemma 5.2]\label{Large-v+i}
Let $T$ be an $(r \times c, v)$-NTA with $v \geq rc - c \cdot \min(r, c - 1)/2$. For each $0 < i \leq rc - v$, there exists an $(r \times c, v + i)$-NTA.
\end{lem}

By Lemma \ref{small-c-v} and Lemma \ref{Large-v+i}, the following result can be immediately obtained.
\begin{lem}\label{v-rc}
There exists a $(4 \times c, v)$-NTA for any $13\leq c\leq 25$ and $2c\leq v \leq 4c $.
\end{lem}

\begin{lem}[\!\!\cite{Gordeev2026}, Lemma 5.5]\label{Large-v}
There exists an $(r \times c, v)$-NTA whenever $v \geq rc - \frac{c}{2}$.
\end{lem}

\vskip 6pt
\textbf{Proof of Theorem \ref{MainTH}: }\
(1) For all integers $c\leq 12$ and $v < 4c - \frac{c}{2}$, a $(4 \times c, v)$-NTA exists except for $(c,v)\in\{ (4,9),(5,7),(5,10),(6,8),(7,9),(10,12),(11,13)\}$. In fact, Appendix B-Table B.2 of \cite{Gordeev2026} gives existence
in exactly the positive entries of the table; the only zero entries in the admissible range are the seven exceptional pairs listed above.
For $c\leq 12$ and $v \geq 4c - \frac{c}{2}$,
the existence of a $(4 \times c, v)$-NTA follows directly from Lemma \ref{Large-v}.

(2) For $13\leq c\leq 25$ and $v \leq 4c - \frac{c}{2}$, the existence of a $(4 \times c, v)$-NTA is ensured by Lemma \ref{small1-c-v}, Lemma \ref{small-c-v} and Lemma \ref{v-rc}. For $13\leq c\leq 25$ and $v \geq 4c - \frac{c}{2}$, by Lemma \ref{Large-v}, there exists a $(4 \times c, v)$-NTA.

(3) For $26\leq c\leq v$, we proceed by induction on $c$ , where the cases $12\leq c\leq25$ serve as the base case. Note that for any $(4\times c, v)$-NTA with $c\geq 13$, we have $\lambda_{cc} = \frac{4(\lambda_{rc} - 1)}{c - 1} \leq \frac{12}{c - 1} \leq 1$.

Let $c \geq 26$. If $v \geq 4c - \frac{c}{2}=\frac{7c}{2}$, the existence of a $(4 \times c, v)$-NTA follows from Lemma \ref{Large-v}.

For $c \leq v \leq c + 3$, we consider two near triple arrays with parameters $(4 \times 13, 13)$-NTA and $(4 \times (c-13), v-13)$-NTA. The former satisfies $e = 4$, $\lambda_{rr} = 13 \in \mathbb{Z}$ and $\lambda_{cc} = 1$, while the latter has $e^+ = 4$ and $\lambda_{cc} \leq 1$. Applying Lemma \ref{T1+T2}, we obtain the desired $(4 \times c, v)$-NTA.

For $c + 4 \leq v \leq 2c-6$, we consider near triple arrays with parameters $(4 \times 9, 12)$-NTA and $(4 \times(c-9), v-12)$-NTA. The former has $e = 3$, $\lambda_{rr} = 6 \in \mathbb{Z}$ and $\lambda_{cc} = 1$, and the latter has either $e^- = 3$ or $e^- = 2$ and $\lambda_{cc} \leq 1$. Therefore, there exists a $(4 \times c, v)$-NTA by Lemma \ref{T1+T2}.

For $2c - 6 < v < \frac{7c}{2}$, we consider near triple arrays with parameters $(4 \times 6, 12)$-NTA and $(4 \times(c-6), v-12)$-NTA. The first one has $e = 2$, $\lambda_{rr} = 2 \in \mathbb{Z}$ and $\lambda_{cc} < 1$, the second has either $e^+ = 2$ or $e^- = 2$ and $\lambda_{cc} \leq 1$. Applying Lemma \ref{T1+T2}, we obtain the desired $(4 \times c, v)$-NTA. \qed

\begin{lem}[\!\!\cite{Gordeev2026}, Lemma 5.9]\label{Dual}
If there exists an $(r \times n, n)$-NTA, then there exists an $((n-r) \times n, n)$-NTA and an $((n-r) \times (n-1), n)$-NTA.
\end{lem}

\begin{cor}
There exist $((n-4) \times n, n)$-NTAs and $((n-4) \times (n-1), n)$-NTAs for any $n > 4$.
\end{cor}

\begin{proof}
This follows immediately from Theorem \ref{MainTH} and Lemma \ref{Dual} with $r=4$ and $c = v = n$.
\end{proof}

\section{Conclusions}

In this paper, we have completely resolved the existence of near triple arrays with four rows. It is natural to pose the next research problem: determining the existence of $(5\times c,v)$-NTA or, more generally, $(r\times c,v)$-NTA for $r\geq 6$. The existence of $(5\times c,v)$-NTA  for sufficiently large $c$
may potentially be proved using a similar approach to that in the proof of Theorem \ref{MainTH}. By Theorem 6.4 in \cite{Gordeev2026}, there is no $(5\times 19,21)$-NTA. Let $v \ge k \ge 2$. A $(v,k,\lambda)$ \emph{packing} is a pair \((X, \mathcal{B})\), where \(X\) is a \(v\)-set of elements (\emph{points}) and \(\mathcal{B}\) is a collection of \(k\)-subsets of \(X\) (\emph{blocks}), such that every pair of distinct points occurs in at most \(\lambda\) blocks in \(\mathcal{B}\). If a $(5\times 22,22)$-NTA exists, then treating its columns as blocks yields a $(22,5,1)$ packing consisting of $22$ blocks. However, there is no
$(22,5,1)$ packing consisting of $22$ blocks, see Theorem 5 in \cite{DFGMP-2016}, and then we have the following result.

\begin{lem}
There is no $(5 \times 22, 22)$-NTA.
\end{lem}

Therefore, to fully settle the existence of near triple arrays with five rows, the $(5 \times 22, 22)$-NTA cannot serve as the base of this recurrence. By Lemma \ref{4nn}, a $(5 \times 23, 23)$-NTA exists and may serve as one of the initial cases for a prospective recurrence. However, additional base cases and the conditions needed to cover the intended five-row parameter range remain to be established.

\vskip 12pt
\noindent\textbf{Declaration of competing interest}

The authors declare that there are no conflicts of interest.

\vskip 12pt
\noindent\textbf{Data availability}

No data was used for the research described in the article.

\vskip 12pt
\noindent\textbf{Acknowledgements}

This work was supported by National Natural Science Foundation of China (Grant No. 12471245) and Natural Science Foundation of Henan Province (Grant No. 262300421846).


\newpage

\begin{center}
\textbf{Appendix A}
\end{center}

\footnotesize{}
{
\renewcommand{\arraystretch}{1.1}  
\setlength{\tabcolsep}{3.5pt}         

A $(16, \{3,4\}, 10, \{0,1\}, \{3,4\}: 4\times 13)$-NTA:
\[
\begin{array}{|ccccccccccccc|}
\hline
1 & 4 & 11 & 0 & 12 & 10 & 6 & 15 & 5 & 2 & 3 & 13 & 14 \\
2 & 10 & 0 & 9 & 1 & 5 & 11 & 4 & 7 & 6 & 12 & 8 & 3 \\
0 & 7 & 5 & 6 & 4 & 14 & 1 & 2 & 13 & 8 & 15 & 3 & 9 \\
3 & 0 & 8 & 12 & 13 & 1 & 15 & 9 & 2 & 14 & 7 & 10 & 11 \\
\hline
\end{array}
\]

A $(17, \{3,4\}, 9, \{0,1\}, \{3,4\}: 4\times 13)$-NTA:
\[
\begin{array}{|ccccccccccccc|}
\hline
5 & 0 & 2 & 4 & 12 & 7 & 3 & 1 & 8 & 6 & 10 & 11 & 9 \\
0 & 3 & 15 & 6 & 16 & 14 & 13 & 8 & 2 & 1 & 4 & 5 & 7 \\
9 & 10 & 0 & 11 & 2 & 4 & 12 & 13 & 14 & 15 & 16 & 3 & 1 \\
13 & 14 & 7 & 0 & 5 & 12 & 6 & 11 & 9 & 10 & 8 & 15 & 16 \\\hline
\end{array}
\]

An $(18, \{2,3\}, \{8,9\}, \{0,1\}, \{2,3\}: 4\times 13)$-NTA:
\[
\begin{array}{|ccccccccccccc|}
\hline
0 & 11 & 2 & 7 & 6 & 10 & 9 & 13 & 12 & 8 & 3 & 4 & 5 \\
14 & 0 & 9 & 3 & 4 & 2 & 15 & 16 & 7 & 5 & 6 & 8 & 17 \\
10 & 15 & 1 & 4 & 14 & 16 & 3 & 5 & 2 & 12 & 17 & 13 & 11 \\
13 & 16 & 6 & 1 & 11 & 8 & 12 & 7 & 17 & 14 & 10 & 15 & 9 \\
\hline
\end{array}
\]

A $(19, \{2,3\}, \{7,8\}, \{0,1\}, \{2,3\}: 4\times 13)$-NTA:
\[
\begin{array}{|ccccccccccccc|}
\hline
0 & 12 & 8 & 1 & 7 & 2 & 10 & 6 & 13 & 11 & 5 & 14 & 9 \\
18 & 0 & 15 & 9 & 5 & 16 & 17 & 3 & 4 & 8 & 10 & 6 & 7 \\
11 & 16 & 1 & 13 & 15 & 6 & 3 & 12 & 5 & 7 & 18 & 17 & 14 \\
15 & 13 & 10 & 17 & 2 & 18 & 11 & 9 & 14 & 4 & 12 & 8 & 16 \\
\hline
\end{array}
\]

A $(20, \{2,3\}, \{7,8\}, \{0,1\}, \{2,3\}: 4\times 13)$-NTA:
\[
\begin{array}{|ccccccccccccc|}
\hline
14 & 0 & 1 & 9 & 11 & 10 & 8 & 12 & 13 & 4 & 16 & 5 & 15 \\
0 & 18 & 7 & 1 & 8 & 19 & 13 & 11 & 17 & 12 & 9 & 10 & 6 \\
6 & 15 & 14 & 17 & 16 & 2 & 3 & 9 & 4 & 8 & 19 & 18 & 10 \\
17 & 7 & 19 & 15 & 2 & 12 & 14 & 3 & 16 & 18 & 5 & 13 & 11 \\
\hline
\end{array}
\]

A $(21, \{2,3\}, \{6,7\}, \{0,1\}, \{2,3\}: 4\times 13)$-NTA:
\[
\begin{array}{|ccccccccccccc|}
\hline
6 & 0 & 1 & 15 & 2 & 17 & 7 & 18 & 14 & 11 & 12 & 13 & 16 \\
0 & 8 & 14 & 9 & 19 & 20 & 3 & 11 & 13 & 4 & 15 & 16 & 12 \\
19 & 20 & 10 & 1 & 11 & 12 & 13 & 3 & 18 & 8 & 5 & 17 & 6 \\
17 & 18 & 20 & 19 & 10 & 2 & 15 & 16 & 4 & 7 & 14 & 5 & 9 \\
\hline
\end{array}
\]

A $(22, \{2,3\}, \{6,7\}, \{0,1\}, \{2,3\}: 4\times 13)$-NTA:
\[
\begin{array}{|ccccccccccccc|}
\hline
0 & 16 & 1 & 14 & 2 & 8 & 18 & 15 & 7 & 12 & 6 & 19 & 17 \\
15 & 0 & 4 & 10 & 9 & 3 & 20 & 6 & 12 & 17 & 21 & 16 & 14 \\
20 & 5 & 15 & 1 & 21 & 13 & 3 & 11 & 19 & 9 & 7 & 14 & 18 \\
2 & 21 & 18 & 20 & 19 & 17 & 16 & 8 & 10 & 11 & 13 & 4 & 5 \\
\hline
\end{array}
\]

A $(23, \{2,3\}, \{5,6\}, \{0,1\}, \{2,3\}: 4\times 13)$-NTA:
\[
\begin{array}{|ccccccccccccc|}
\hline
0 & 21 & 18 & 1 & 2 & 19 & 20 & 8 & 6 & 13 & 7 & 17 & 12 \\
17 & 0 & 4 & 10 & 18 & 22 & 14 & 3 & 9 & 6 & 19 & 15 & 20 \\
5 & 16 & 1 & 21 & 9 & 14 & 3 & 18 & 11 & 22 & 12 & 7 & 17 \\
22 & 19 & 15 & 20 & 21 & 2 & 13 & 5 & 8 & 16 & 10 & 11 & 4 \\
\hline
\end{array}
\]

A $(17, \{3,4\}, 11, \{0,1\}, \{3,4\}: 4\times 14)$-NTA:
\[
\begin{array}{|cccccccccccccc|}
\hline
0 & 4 & 5 & 12 & 1 & 3 & 13 & 7 & 10 & 9 & 2 & 8 & 6 & 11 \\
3 & 0 & 15 & 9 & 4 & 1 & 8 & 14 & 2 & 5 & 6 & 7 & 16 & 10 \\
14 & 7 & 0 & 16 & 15 & 5 & 6 & 1 & 13 & 11 & 12 & 2 & 4 & 3 \\
8 & 11 & 13 & 0 & 9 & 12 & 1 & 10 & 4 & 2 & 14 & 16 & 3 & 15 \\
\hline
\end{array}
\]

An $(18, \{3,4\}, 10, \{0,1\}, \{3,4\}: 4\times 14)$-NTA:
\[
\begin{array}{|cccccccccccccc|}
\hline
0 & 13 & 2 & 1 & 8 & 3 & 9 & 4 & 11 & 5 & 6 & 10 & 12 & 7 \\
6 & 1 & 0 & 2 & 3 & 7 & 4 & 8 & 5 & 9 & 17 & 16 & 15 & 14 \\
11 & 15 & 10 & 16 & 14 & 0 & 1 & 12 & 17 & 13 & 2 & 4 & 3 & 5 \\
14 & 7 & 15 & 12 & 1 & 17 & 11 & 0 & 8 & 16 & 13 & 6 & 9 & 10 \\
\hline
\end{array}
\]

A $(19, \{2,3\}, \{9,10\}, \{0,1\}, \{2,3\}: 4\times 14)$-NTA:
\[
\begin{array}{|cccccccccccccc|}
\hline
0 & 13 & 11 & 6 & 7 & 12 & 2 & 10 & 9 & 1 & 8 & 5 & 3 & 4 \\
14 & 0 & 15 & 16 & 2 & 7 & 18 & 1 & 8 & 5 & 3 & 4 & 6 & 17 \\
9 & 12 & 1 & 2 & 13 & 3 & 11 & 16 & 4 & 17 & 14 & 15 & 18 & 10 \\
16 & 17 & 6 & 12 & 14 & 15 & 5 & 8 & 18 & 9 & 11 & 13 & 10 & 7 \\
\hline
\end{array}
\]

A $(20, \{2,3\}, \{8,9\}, \{0,1\}, \{2,3\}: 4\times 14)$-NTA:
\[
\begin{array}{|cccccccccccccc|}
\hline
8 & 0 & 11 & 1 & 10 & 13 & 14 & 7 & 9 & 5 & 6 & 15 & 4 & 12 \\
7 & 9 & 18 & 5 & 19 & 2 & 16 & 3 & 4 & 11 & 10 & 6 & 17 & 8 \\
0 & 13 & 4 & 16 & 2 & 7 & 6 & 12 & 15 & 19 & 18 & 17 & 14 & 5 \\
16 & 17 & 1 & 15 & 12 & 11 & 3 & 9 & 19 & 14 & 13 & 8 & 10 & 18 \\
\hline
\end{array}
\]

A $(21, \{2,3\}, \{8,9\}, \{0,1\}, \{2,3\}: 4\times 14)$-NTA:
\[
\begin{array}{|cccccccccccccc|}
\hline
0 & 14 & 1 & 11 & 2 & 15 & 9 & 10 & 13 & 7 & 16 & 12 & 8 & 6 \\
18 & 0 & 7 & 20 & 17 & 19 & 12 & 3 & 10 & 4 & 9 & 8 & 6 & 11 \\
16 & 17 & 19 & 1 & 7 & 8 & 3 & 20 & 9 & 18 & 5 & 13 & 14 & 15 \\
13 & 19 & 10 & 17 & 12 & 2 & 15 & 16 & 4 & 14 & 11 & 5 & 20 & 18 \\
\hline
\end{array}
\]

A $(22, \{2,3\}, \{7,8\}, \{0,1\}, \{2,3\}: 4\times 14)$-NTA:
\[
\begin{array}{|cccccccccccccc|}
\hline
7 & 0 & 1 & 18 & 10 & 2 & 17 & 16 & 12 & 14 & 11 & 13 & 6 & 15 \\
0 & 15 & 12 & 8 & 9 & 13 & 11 & 3 & 19 & 4 & 21 & 10 & 14 & 20 \\
17 & 8 & 20 & 1 & 21 & 11 & 3 & 19 & 18 & 10 & 5 & 16 & 12 & 6 \\
19 & 16 & 7 & 13 & 2 & 20 & 9 & 14 & 4 & 17 & 15 & 5 & 21 & 18 \\
\hline
\end{array}
\]

A $(23, \{2,3\}, \{7,8\}, \{0,1\}, \{2,3\}: 4\times 14)$-NTA:
\[
\begin{array}{|cccccccccccccc|}
\hline
0 & 8 & 7 & 1 & 2 & 19 & 16 & 15 & 14 & 17 & 13 & 12 & 6 & 18 \\
20 & 0 & 22 & 13 & 14 & 10 & 9 & 3 & 4 & 12 & 15 & 16 & 21 & 6 \\
7 & 14 & 1 & 20 & 9 & 13 & 3 & 18 & 11 & 21 & 5 & 22 & 19 & 17 \\
19 & 21 & 15 & 8 & 20 & 2 & 17 & 10 & 18 & 4 & 11 & 5 & 16 & 22 \\
\hline
\end{array}
\]

An $(18, \{3,4\}, 12, \{0,1\}, \{3,4\}: 4\times 15)$-NTA:
\[
\begin{array}{|ccccccccccccccc|}
\hline
3 & 0 & 10 & 6 & 1 & 9 & 13 & 8 & 14 & 4 & 5 & 2 & 7 & 12 & 11 \\
9 & 4 & 16 & 0 & 7 & 1 & 15 & 10 & 8 & 11 & 2 & 6 & 17 & 3 & 5 \\
0 & 7 & 5 & 14 & 16 & 6 & 1 & 17 & 15 & 2 & 12 & 13 & 3 & 4 & 8 \\
13 & 15 & 0 & 11 & 12 & 4 & 5 & 1 & 2 & 17 & 9 & 16 & 14 & 10 & 3 \\
\hline
\end{array}
\]

A $(19, \{3,4\}, 11, \{0,1\}, \{3,4\}: 4\times 15)$-NTA:
\[
\begin{array}{|ccccccccccccccc|}
\hline
0 & 15 & 14 & 4 & 1 & 17 & 11 & 16 & 5 & 12 & 6 & 2 & 18 & 13 & 3 \\
11 & 0 & 9 & 8 & 10 & 1 & 3 & 13 & 2 & 17 & 7 & 14 & 4 & 18 & 16 \\
6 & 7 & 0 & 16 & 14 & 5 & 1 & 9 & 11 & 2 & 18 & 15 & 12 & 8 & 10 \\
17 & 13 & 3 & 0 & 6 & 7 & 15 & 1 & 10 & 8 & 2 & 4 & 9 & 5 & 12 \\
\hline
\end{array}
\]

A $(19, \{3,4\}, 13, \{0,1\}, \{3,4\}: 4\times 16)$-NTA:
\[
\begin{array}{|cccccccccccccccc|}
\hline
4 & 12 & 6 & 0 & 7 & 11 & 9 & 1 & 8 & 5 & 2 & 10 & 14 & 3 & 15 & 13 \\
0 & 5 & 8 & 7 & 1 & 18 & 17 & 10 & 4 & 9 & 6 & 2 & 11 & 16 & 3 & 12 \\
9 & 0 & 16 & 13 & 4 & 5 & 1 & 8 & 2 & 15 & 18 & 14 & 3 & 7 & 6 & 17 \\
18 & 14 & 0 & 11 & 15 & 1 & 6 & 13 & 17 & 2 & 12 & 16 & 4 & 5 & 10 & 3 \\
\hline
\end{array}
\]

A $(20, \{3,4\}, 12, \{0,1\}, \{3,4\}: 4\times 16)$-NTA:
\[
\begin{array}{|cccccccccccccccc|}
\hline
10 & 6 & 14 & 0 & 7 & 13 & 1 & 9 & 2 & 11 & 12 & 4 & 15 & 5 & 8 & 3 \\
0 & 8 & 4 & 9 & 17 & 1 & 19 & 6 & 18 & 5 & 7 & 2 & 10 & 3 & 16 & 11 \\
5 & 0 & 18 & 19 & 1 & 16 & 4 & 14 & 15 & 17 & 2 & 13 & 6 & 12 & 3 & 7 \\
13 & 17 & 0 & 12 & 15 & 11 & 10 & 1 & 8 & 2 & 16 & 9 & 3 & 18 & 14 & 19 \\
\hline
\end{array}
\]

A $(21, \{3,4\}, 11, \{0,1\}, \{3,4\}: 4\times 16)$-NTA:
\[
\begin{array}{|cccccccccccccccc|}
\hline
6 & 8 & 0 & 14 & 15 & 3 & 7 & 12 & 2 & 1 & 11 & 10 & 4 & 13 & 9 & 5 \\
0 & 17 & 19 & 18 & 1 & 10 & 2 & 4 & 9 & 20 & 8 & 5 & 16 & 6 & 3 & 7 \\
2 & 3 & 4 & 0 & 19 & 16 & 20 & 18 & 17 & 12 & 5 & 15 & 14 & 1 & 11 & 13 \\
12 & 0 & 11 & 7 & 9 & 13 & 15 & 8 & 14 & 10 & 16 & 18 & 6 & 17 & 20 & 19 \\
\hline
\end{array}
\]

A $(23, \{2,3\}, \{9,10\}, \{0,1\}, \{2,3\}: 4\times 16)$-NTA:
\[
\begin{array}{|cccccccccccccccc|}
\hline
20 & 4 & 8 & 12 & 16 & 15 & 19 & 3 & 7 & 11 & 18 & 2 & 6 & 10 & 14 & 0 \\
1 & 5 & 9 & 13 & 17 & 0 & 4 & 8 & 12 & 16 & 7 & 11 & 15 & 19 & 22 & 20  \\
2 & 21 & 10 & 14 & 18 & 5 & 9 & 13 & 17 & 1 & 0 & 4 & 8 & 12 & 16 & 6 \\
3 & 7 & 11 & 15 & 19 & 10 & 14 & 22 & 2 & 6 & 9 & 13 & 17 & 1 & 5 & 21 \\
\hline
\end{array}
\]

A $(20, \{3,4\}, 14, \{0,1\}, \{3,4\}: 4\times 17)$-NTA:
\[
\begin{array}{|ccccccccccccccccc|}
\hline
16 & 5 & 0 & 15 & 12 & 14 & 6 & 1 & 2 & 10 & 11 & 8 & 4 & 3 & 9 & 7 & 13 \\
0 & 18 & 12 & 7 & 4 & 1 & 8 & 11 & 9 & 19 & 2 & 13 & 3 & 5 & 6 & 10 & 17 \\
4 & 8 & 6 & 0 & 10 & 17 & 1 & 7 & 14 & 5 & 15 & 2 & 18 & 16 & 19 & 3 & 9 \\
11 & 0 & 17 & 19 & 1 & 5 & 16 & 18 & 4 & 2 & 6 & 7 & 13 & 12 & 3 & 14 & 15 \\
\hline
\end{array}
\]

A $(21, \{3,4\}, 13, \{0,1\}, \{3,4\}: 4\times 17)$-NTA:
\[
\begin{array}{|ccccccccccccccccc|}
\hline
0 & 6 & 5 & 14 & 11 & 9 & 16 & 1 & 4 & 7 & 2 & 8 & 13 & 3 & 15 & 12 & 10 \\
18 & 0 & 19 & 8 & 1 & 6 & 7 & 12 & 17 & 9 & 5 & 2 & 4 & 10 & 11 & 3 & 20 \\
4 & 17 & 0 & 20 & 14 & 1 & 18 & 8 & 2 & 15 & 13 & 16 & 19 & 6 & 3 & 7 & 5 \\
9 & 16 & 11 & 0 & 4 & 19 & 1 & 13 & 12 & 2 & 18 & 10 & 3 & 14 & 17 & 20 & 15 \\
\hline
\end{array}
\]

A $(22, \{3,4\}, 12, \{0,1\}, \{3,4\}: 4\times 17)$-NTA:
\[
\begin{array}{|ccccccccccccccccc|}
\hline
0 & 2 & 13 & 10 & 1 & 4 & 16 & 12 & 15 & 8 & 14 & 11 & 7 & 6 & 3 & 5 & 9 \\
6 & 11 & 7 & 0 & 3 & 21 & 1 & 17 & 9 & 4 & 20 & 5 & 18 & 19 & 10 & 8 & 2 \\
14 & 17 & 0 & 15 & 20 & 1 & 5 & 6 & 3 & 19 & 2 & 21 & 4 & 13 & 16 & 18 & 12 \\
8 & 0 & 20 & 19 & 11 & 14 & 7 & 1 & 17 & 16 & 10 & 13 & 12 & 9 & 18 & 15 & 21 \\
\hline
\end{array}
\]

A $(23, \{2,3\}, \{11,12\}, \{0,1\}, \{2,3\}: 4\times 17)$-NTA:
\[
\begin{array}{|ccccccccccccccccc|}
\hline
11 & 0 & 7 & 12 & 3 & 15 & 10 & 2 & 13 & 1 & 16 & 5 & 4 & 9 & 6 & 8 & 14 \\
0 & 18 & 1 & 2 & 19 & 21 & 4 & 6 & 8 & 22 & 10 & 9 & 7 & 17 & 5 & 3 & 20 \\
17 & 13 & 20 & 18 & 14 & 1 & 19 & 16 & 5 & 12 & 3 & 11 & 15 & 2 & 21 & 22 & 4 \\
16 & 21 & 11 & 8 & 7 & 9 & 13 & 22 & 17 & 10 & 20 & 19 & 18 & 14 & 12 & 15 & 6 \\
\hline
\end{array}
\]

A $(21, \{3,4\}, 15, \{0,1\}, \{3,4\}: 4\times 18)$-NTA:
\[
\begin{array}{|cccccccccccccccccc|}
\hline
12 & 9 & 0 & 14 & 13 & 4 & 1 & 8 & 11 & 2 & 7 & 17 & 5 & 6 & 15 & 3 & 10 & 16 \\
4 & 6 & 13 & 0 & 9 & 1 & 7 & 18 & 5 & 10 & 2 & 8 & 3 & 12 & 11 & 19 & 14 & 20 \\
0 & 17 & 7 & 8 & 5 & 15 & 10 & 1 & 19 & 6 & 9 & 2 & 18 & 20 & 3 & 16 & 4 & 11 \\
5 & 0 & 18 & 15 & 1 & 6 & 17 & 12 & 2 & 16 & 20 & 13 & 14 & 3 & 7 & 8 & 19 & 4 \\
\hline
\end{array}
\]

A $(22, \{3,4\}, 14, \{0,1\}, \{3,4\}: 4\times 18)$-NTA:
\[
\begin{array}{|cccccccccccccccccc|}
\hline
10 & 0 & 14 & 16 & 4 & 1 & 9 & 13 & 6 & 11 & 2 & 15 & 7 & 5 & 12 & 3 & 8 & 17 \\
4 & 12 & 0 & 8 & 7 & 19 & 1 & 18 & 10 & 5 & 20 & 2 & 3 & 13 & 6 & 9 & 11 & 21 \\
15 & 18 & 6 & 0 & 1 & 5 & 14 & 8 & 21 & 2 & 7 & 9 & 19 & 3 & 17 & 16 & 20 & 4 \\
0 & 5 & 19 & 21 & 16 & 15 & 12 & 1 & 2 & 17 & 14 & 18 & 11 & 20 & 3 & 10 & 4 & 13 \\
\hline
\end{array}
\]

A $(23, \{3,4\}, 13, \{0,1\}, \{3,4\}: 4\times 18)$-NTA:
\[
\begin{array}{|cccccccccccccccccc|}
\hline
14 & 0 & 11 & 5 & 7 & 3 & 17 & 1 & 2 & 9 & 4 & 12 & 10 & 6 & 16 & 15 & 8 & 13 \\
18 & 6 & 0 & 12 & 1 & 8 & 10 & 22 & 11 & 7 & 21 & 2 & 19 & 9 & 5 & 3 & 20 & 4 \\
7 & 15 & 3 & 0 & 20 & 16 & 1 & 6 & 13 & 17 & 2 & 14 & 4 & 21 & 22 & 19 & 5 & 18 \\
0 & 10 & 21 & 19 & 15 & 1 & 18 & 13 & 20 & 2 & 8 & 22 & 16 & 14 & 11 & 9 & 17 & 12 \\
\hline
\end{array}
\]

A $(22, \{3,4\}, 16, \{0,1\}, \{3,4\}: 4\times 19)$-NTA:
\[
\begin{array}{|ccccccccccccccccccc|}
\hline
8 & 0 & 18 & 17 & 1 & 10 & 9 & 15 & 2 & 5 & 6 & 7 & 4 & 11 & 12 & 3 & 13 & 16 & 14 \\
20 & 15 & 0 & 11 & 14 & 5 & 1 & 19 & 4 & 2 & 8 & 12 & 9 & 3 & 6 & 7 & 21 & 13 & 10 \\
4 & 5 & 10 & 0 & 21 & 1 & 6 & 7 & 19 & 16 & 2 & 9 & 18 & 20 & 3 & 8 & 11 & 12 & 17 \\
0 & 9 & 6 & 7 & 4 & 8 & 16 & 1 & 13 & 14 & 15 & 2 & 3 & 5 & 17 & 21 & 18 & 20 & 19\\
\hline
\end{array}
\]

A $(23, \{3,4\}, 15, \{0,1\}, \{3,4\}: 4\times 19)$-NTA:
\[
\begin{array}{|ccccccccccccccccccc|}
\hline
11 & 16 & 6 & 0 & 10 & 12 & 1 & 7 & 8 & 2 & 9 & 17 & 15 & 5 & 3 & 13 & 14 & 4 & 18 \\
20 & 9 & 0 & 7 & 22 & 5 & 8 & 1 & 2 & 13 & 6 & 10 & 11 & 14 & 21 & 3 & 4 & 19 & 12 \\
0 & 5 & 22 & 18 & 1 & 17 & 6 & 21 & 16 & 15 & 2 & 20 & 19 & 3 & 10 & 9 & 8 & 7 & 4 \\
4 & 0 & 12 & 13 & 11 & 1 & 20 & 14 & 22 & 5 & 19 & 2 & 3 & 18 & 6 & 17 & 15 & 16 & 21 \\
\hline
\end{array}
\]

A $(24, \{3,4\}, 14, \{0,1\}, \{3,4\}: 4\times 19)$-NTA:
\[
\begin{array}{|ccccccccccccccccccc|}
\hline
11 & 12 & 0 & 14 & 8 & 1 & 17 & 4 & 2 & 16 & 7 & 10 & 6 & 18 & 3 & 13 & 5 & 15 & 9 \\
6 & 23 & 13 & 0 & 1 & 22 & 5 & 21 & 9 & 12 & 11 & 2 & 19 & 10 & 8 & 3 & 20 & 7 & 4 \\
0 & 17 & 7 & 8 & 18 & 16 & 1 & 14 & 6 & 19 & 2 & 5 & 3 & 4 & 21 & 23 & 15 & 22 & 20\\
21 & 0 & 19 & 20 & 23 & 10 & 13 & 1 & 22 & 2 & 18 & 14 & 15 & 3 & 9 & 16 & 11 & 12 & 17 \\\hline
\end{array}
\]

A $(23, \{3,4\}, 17, \{0,1\}, \{3,4\}: 4\times 20)$-NTA:
\[
\begin{array}{|cccccccccccccccccccc|}
\hline
5 & 10 & 12 & 0 & 1 & 11 & 16 & 17 & 2 & 6 & 9 & 18 & 3 & 15 & 7 & 8 & 4 & 19 & 14 & 13 \\
9 & 6 & 0 & 13 & 10 & 1 & 7 & 8 & 11 & 14 & 15 & 2 & 5 & 22 & 3 & 21 & 16 & 12 & 4 & 20 \\
0 & 17 & 21 & 19 & 20 & 9 & 1 & 12 & 5 & 2 & 7 & 8 & 13 & 6 & 11 & 3 & 18 & 10 & 22 & 4 \\
14 & 0 & 7 & 8 & 5 & 6 & 22 & 1 & 21 & 19 & 2 & 10 & 18 & 3 & 20 & 16 & 9 & 4 & 17 & 15 \\
\hline
\end{array}
\]

A $(24, \{3,4\}, 16, \{0,1\}, \{3,4\}: 4\times 20)$-NTA:
\[
\begin{array}{|cccccccccccccccccccc|}
\hline
5 & 0 & 13 & 18 & 10 & 1 & 7 & 8 & 14 & 2 & 19 & 11 & 9 & 6 & 17 & 3 & 4 & 12 & 15 & 16 \\
12 & 21 & 0 & 10 & 5 & 6 & 1 & 22 & 2 & 23 & 7 & 15 & 3 & 11 & 14 & 8 & 13 & 20 & 4 & 9 \\
16 & 9 & 7 & 0 & 1 & 19 & 18 & 17 & 8 & 6 & 2 & 21 & 5 & 20 & 3 & 23 & 11 & 4 & 10 & 22 \\
0 & 6 & 22 & 20 & 13 & 14 & 21 & 1 & 5 & 16 & 12 & 2 & 15 & 3 & 7 & 18 & 19 & 17 & 23 & 4 \\\hline
\end{array}
\]

A $(25, \{3,4\}, 15, \{0,1\}, \{3,4\}: 4\times 20)$-NTA:
\[
\begin{array}{|cccccccccccccccccccc|}
\hline
12 & 7 & 0 & 14 & 18 & 1 & 13 & 15 & 17 & 2 & 16 & 8 & 3 & 6 & 10 & 5 & 19 & 11 & 9 & 4 \\
6 & 21 & 13 & 0 & 20 & 11 & 1 & 22 & 10 & 23 & 5 & 2 & 9 & 24 & 7 & 3 & 8 & 4 & 12 & 14 \\
15 & 0 &  16 & 5 & 1 & 7 & 8 & 9 & 2 & 6 & 21 & 18 & 20 & 19 & 3 & 22 & 4 & 17 & 23 & 24 \\
0 & 17 & 22 & 23 & 12 & 19 & 21 & 1 & 24 & 11 & 2 & 14 & 13 & 3 & 16 & 18 & 10 & 20 & 4 & 15 \\
\hline
\end{array}
\]

A $(24, \{3,4\}, 18, \{0,1\}, \{3,4\}: 4\times 21)$-NTA:
\[
\begin{array}{|ccccccccccccccccccccc|}
\hline
5 & 16 & 7 & 0 & 1 & 6 & 10 & 11 & 2 & 12 & 9 & 8 & 17 & 18 & 3 & 14 & 15 & 19 & 20 & 4 & 13 \\
0 & 6 & 11 & 13 & 14 & 9 & 1 & 17 & 5 & 23 & 2 & 21 & 3 & 15 & 7 & 8 & 4 & 10 & 12 & 16 & 22 \\
12 & 10 & 0 & 20 & 21 & 1 & 7 & 8 & 18 & 6 & 14 & 2 & 5 & 3 & 13 & 23 & 9 & 4 & 11 & 22 & 19 \\
9 & 0 & 22 & 8 & 5 & 19 & 23 & 1 & 11 & 2 & 7 & 10 & 20 & 6 & 16 & 3 & 21 & 17 & 4 & 18 & 15\\
\hline
\end{array}
\]
}


\begin{thebibliography}{99}\addtolength{\itemsep}{-2.05ex}

\bibitem{Agrawal1966}
H. Agrawal, Some methods of construction of designs for two-way elimination of
heterogeneity, J. Amer. Statist. Assoc. 61 (1966) 1153-1171.

\bibitem{Bailey2017}
R. A. Bailey, Relations among partitions, In: A. Claesson, M. Dukes, S. Kitaev,
D. Manlove, and K. Meeks(Eds.), Surveys in Combinatorics 2017. London Mathematical Society Lecture Note Series. Cambridge University Press, 2017: 1-86.

\bibitem{Bagchi-Shah1989}
S. Bagchi, K. R. Shah, On the optimality of a class of row-column designs, J. Statist. Plann. Inference 23(3) (1989) 397-402.

\bibitem{Bagchi1998}
S. Bagchi, On two-way designs, Graphs Combin. 14(4) (1998) 313-319.

\bibitem{Bagchi2026}
S. Bagchi, B. Bagchi, Triple arrays from ovals in finite projective
planes. arXiv:2607.20275, 2026.

\bibitem{Bofill2004}
P. Bofill, C. Torras, MBMUDs: a combinatorial extension of BIBDs showing good optimality behaviour, J. Statist. Plann. Inference 124(1) (2004) 185-204.

\bibitem{Bose1942}
R. C. Bose, An affine analogue of Singer's theorem, J. Indian Math. Soc. 6 (1942)
1-15.

\bibitem{Buratti2021}
M. Buratti, D. Stinson, New results on modular Golomb rulers, optical
orthogonal codes and related structures, Ars Math. Contemp. 20 (2021) 1-27.


\bibitem{Handbook}
C. J. Colbourn, J. H. Dinitz, Handbook of Combinatorial Designs, 2nd ed., Chapman $\&{}$ Hall/CRC, 2007.


\bibitem{DFGMP-2016}
A. A. Davydov, G. Faina, M. Giulietti, S. Marcugini, F. Pambianco, On constructions and parameters of symmetric configurations $v_k$. Des. Codes Cryptogr. 80 (2016) 125-147.


\bibitem{Fon-Der-Flaass-1997}
D. G. Fon-Der-Flaass, Arrays of distinct representatives-a very simple NP-complete
problem. Discrete Math. 171(1) (1997) 295-298.

\bibitem{Gordon1994}
D. M. Gordon, The prime power conjecture is true for $n < 2000 000$, Electron. J. Combin. 1 (1994) \#R6.

\bibitem{Gordeev2026}
A. Gordeev, K. Markstr\"{o}m, L. -D. \"{O}hman, Near triple arrays, J. Combin. Theory Ser. A 219 (2026) 106121.

\bibitem{JMSO23}
G. J\"{a}ger, K. Markstr\"{o}m, D. Shcherbak, L. -D. \"{O}hman. Small Youden rectangles, near Youden rectangles, and their connections to other
row-column designs. Discrete Math. Theor. Comput. Sci.
25(1) (2023): \#9.

\bibitem{John1977}
J. A. John, T. J. Mitchell, Optimal incomplete block designs,
J. R. Stat. Soc. Ser. B Stat. Methodol. 39(1) (1977) 39-43.

\bibitem{McSorley2005}
J. P. McSorley, N. C. K. Phillips, W. D. Wallis, J. L. Yucas, Double arrays, triple arrays and balanced grids, Des. Codes Cryptogr. 35(1) (2005) 21-45.

\bibitem{McSorley37}
J. P. McSorley. Double arrays, triple arrays and balanced grids with $v=r+c
-1$. Designs, Codes and Cryptogr. 37(2) (2005) 313-318.

\bibitem{Nilson2015}
T. Nilson, L. -D. \"{O}hman, Triple arrays and Youden squares, Des. Codes Cryptogr. 75(3) (2015) 429-451.

\bibitem{Nilson2017}
T. Nilson, P. J. Cameron, Triple arrays from difference sets, J. Combin. Des. 25(11) (2017) 494-506.

\bibitem{Preece2005}
D. A. Preece, W. D. Wallis, J. L. Yucas, Paley triple arrays, Australas. J.
Combin. 33 (2005) 237-246.

\bibitem{Ruzsa1993}
I. Z. Ruzsa, Solving a linear equation in a set of integers I, Acta Arith. 65 (1993), 259-282.

\bibitem{Seberry1979}
J. R. Seberry, A note on orthogonal Graeco-Latin designs, Ars Combin. 8 (1979) 85-94.

\bibitem{Street1981}
D. J. Street, Graeco-Latin and nested row and column designs.
In: K. L. McAvaney (ed.), Combinatorial Mathematics VIII,
Lecture Notes in Mathematics 884, Springer-Verlag,
Berlin, 1981: 304-313.
\end{thebibliography}
\end{document}